\documentclass{amsart}

\usepackage{graphicx}

\usepackage{times}
\usepackage{amssymb}
\usepackage{amsmath,amsthm}

\usepackage{hyperref}

\usepackage{xypic}
\xyoption{curve}

\makeatletter
\def\@strippedMR{}
\def\@scanforMR#1#2#3\endscan{%
  \ifx#1M\ifx#2R\def\@strippedMR{#3}%
  \else\def\@strippedMR{#1#2#3}%
  \fi\fi}
\renewcommand\MR[1]{\relax
  \ifhmode\unskip\spacefactor3000 \space\fi
  \@scanforMR#1\endscan
  MR\MRhref{\@strippedMR}{\@strippedMR}}
\makeatother 

\usepackage{bm}

\newcommand{\C}{\mathbb{C}}

\newcommand{\N}{\mathbb{N}}
\newcommand{\Q}{\mathbb{Q}}
\newcommand{\R}{\mathbb{R}}
\newcommand{\Z}{\mathbb{Z}}

\newcommand{\f}{{\bm{f}}}
\newcommand{\g}{{\bm{g}}}
\newcommand{\h}{{\bm{h}}}
\newcommand{\bfQ}{{\bm{Q}}}

\newtheorem*{mthm}{Main Theorem}
\newtheorem{thm}{Theorem}[section]
\newtheorem{cor}[thm]{Corollary}
\newtheorem{lem}[thm]{Lemma}

\newtheorem{prop}[thm]{Proposition}

\newtheorem*{claim*}{Claim}

\theoremstyle{definition}
\newtheorem*{defn}{Definition}

\newtheorem*{exam}{Example}
\newtheorem*{ackn}{Acknowledgment}
\newtheorem*{question}{Question}

\DeclareMathOperator{\poly}{Poly}

\DeclareMathOperator{\supp}{supp}
\DeclareMathOperator{\crit}{Crit}

\DeclareMathOperator{\per}{per}
\DeclareMathOperator{\Int}{int}

\DeclareMathOperator{\mult}{mult}

\DeclareMathOperator{\attr}{attr}
\DeclareMathOperator{\rep}{rep}
\DeclareMathOperator{\geom}{geom}

\DeclareMathOperator{\sector}{Sector}
\DeclareMathOperator{\Mis}{Mis}
\DeclareMathOperator{\preper}{Preper}
\DeclareMathOperator{\capt}{cap}

\theoremstyle{remark}
\newtheorem{rem}[thm]{Remark}

\renewcommand{\theenumi}{\rm (\roman{enumi})}

\newcommand{\condI}{(C1)}
\newcommand{\condII}{(C2)}
\newcommand{\condIII}{(C3)}

\newcommand{\semiconj}{conjugate by an irreducible holomorphic correspondence}
\newcommand{\Semiconj}{Conjugate by an irreducible holomorphic correspondence}

\title{Combinatorics and topology of straightening maps II:
Discontinuity}
\author{Hiroyuki Inou}
\address{Department of Mathematics, Kyoto University, Kyoto 606-8502, JAPAN}
\date{\today}

\begin{document}
\maketitle

\begin{abstract}
 We continue the study of straightening maps for the family of
 polynomials of degree $d \ge 3$. 
 The notion of straightening map is
 originally introduced by Douady and Hubbard to study relationship between 
 polynomial-like renormalizations and self-similarity of the Mandelbrot set.
 In the quadratic case, straightening maps are always
 continuous, and this is one of the critical steps to prove the
 Mandelbrot set has small copies in itself.

 On the other hand, for higher degree case, 
 we do not have such a nice self-similar property:
 As expected from an example of a cubic-like family with
 discontinuous straightening map by Douady and Hubbard,
 we prove that the straightening map is discontinuous unless it is of
 disjoint type.
\end{abstract}

\section{Introduction}

Consider the family of monic centered polynomials $\poly(d)$ of degree
$d \ge 2$. The \emph{connectedness locus} $\mathcal{C}(d)$ is the
set of $f \in \poly(d)$ having connected filled Julia set $K(f)$.
In the case of degree two,
$\mathcal{M}=\mathcal{C}(2)$ is the well-known Mandelbrot set.
Douady and Hubbard \cite{Douady-Hubbard-poly-like} proved that there
exist infinitely many small copies of the Mandelbrot set in itself.
In fact, for any $z^2+c_0 \in \mathcal{M}$ such that the critical point
$0$ is periodic, there exist a subset $M' \subset \mathcal{M}$ and a homeomorphism $\chi:M' \to \mathcal{M}$ 
such that $c_0 \in M'$, the boundary of $M'$ is contained in that of
$\mathcal{M}$ and $\chi(c_0)=0$ \cite{Haissinsky-Mandelbrot}.

The map $\chi$ above is an example of {\itshape
straightening maps}.
For a family of polynomial-like mappings parameterized by a complex
manifold $\Lambda$ of degree $d \ge 2$, we can define such a map defined
on the connectedness 
locus of $\Lambda$, taking values in the set of affine conjugacy classes
of polynomials in $\mathcal{C}(d)$.

In the preceding paper \cite{Inou-Kiwi-straightening}, 
we consider straightening maps for families of renormalizable
polynomials of degree $d \ge 3$. 
We characterize the combinatorics of a family of renormalizable
polynomials in terms of rational laminations introduced by Thurston
\cite{Thurston}. 
A \emph{rational lamination} $\lambda_f$ for $f \in \mathcal{C}(d)$
is the landing relation of external rays of rational angles.
Let $\lambda_0$ be a post-critically finite $d$-invariant rational
lamination (equivalently, let $\lambda_0$ be the rational lamination of
a post-critically finite polynomial in $\mathcal{C}(d)$),
and let $\mathcal{C}(\lambda_0)=\{f \in \mathcal{C}(d);\
\lambda_f \supset \lambda_0\}$ denote the set of {\itshape
$\lambda_0$-combinatorially renormalizable} polynomials.
For $f \in \mathcal{C}(\lambda_0)$, we say $f$ is {\itshape
$\lambda_0$-renormalizable} if it has a polynomial-like restriction
whose filled Julia sets are \emph{$\lambda_0$-fibers},
which are continua defined in terms of $\lambda_0$. 
For a $\lambda_0$-renormalizable map $f$, 
we can straighten such a restriction ({\itshape
$\lambda_0$-renormalization}) to get a new polynomial by the
straightening theorem by Douady and Hubbard \cite{Douady-Hubbard-poly-like}.
More precisely, since there might exist several critical points,
we borrow the notion of mapping schema introduced by Milnor
\cite{Milnor-hyp} to describe the dynamics of $\lambda_0$-fibers
containing critical points.
Therefore, the straightening of a $\lambda_0$-renormalization of $f$
is an affine conjugacy class of polynomials over a mapping schema
(say, $T(\lambda_0)$) of $\lambda_0$.

Under this definition, the straightening map is at most
finite-to-one.
By introducing ``markings'' for polynomials and polynomial-like maps, we
can define an \emph{injective} straightening map
$\chi_{\lambda_0}:\mathcal{R}(\lambda_0) \to \mathcal{C}(T(\lambda_0))$,
where $\mathcal{C}(T(\lambda_0))$ is the 
\emph{fiberwise connectedness locus of the family $\poly(T(\lambda_0))$ of
monic centered polynomials over $T(\lambda_0)$}.
We recall these notions and results in the preceding paper in
Section~\ref{sec-comb}.

In this paper, we study discontinuity of straightening maps.
\begin{mthm}
 \label{thm-discont}
 Let $d \ge 3$.
 Assume a $d$-invariant post-critically finite rational lamination
 $\lambda_0$ has a non-trivial Fatou critical relation and its
 straightening map $\chi_{\lambda_0}$ has nonempty domain of definition.

 Then $\chi_{\lambda_0}$ is not continuous. 
 More precisely, $\chi_{\lambda_0}$ is not continuous on any
 neighborhood of any Misiurewicz $\lambda_0$-renormalizable polynomial.
\end{mthm}

We say a polynomial $f$ of degree $d \ge 2$ is \emph{Misiurewicz}
if all critical points are (strictly) preperiodic.
Note that since Misiurewicz maps are quasiconformally rigid (moreover,
they are combinatorially rigid), straightening maps are continuous
\emph{at} those parameters.
It is known that the closure of the set of Misiurewicz polynomials 
coincides with the support of bifurcation measure \cite{Dujardin-Favre}.
Therefore, we may also say
that $\chi_{\lambda_0}$ is not continuous on any open set intersecting
the support of the bifurcation measure.

An equivalent condition for the domain to be nonempty is stated in
\cite{Inou-Kiwi-straightening} (see Proposition~\ref{prop-nonempty}).

It is well-known that straightening maps for quadratic-like families are
always continuous \cite{Douady-Hubbard-poly-like}.
Therefore, we have a complete classification:
\begin{cor}
 Let $d \ge 3$ and let $\lambda_0$ be a $d$-invariant post-critically
 finite rational lamination with $\mathcal{R}(\lambda_0) \ne \emptyset$.
 Then the straightening map $\chi_{\lambda_0}:\mathcal{R}(\lambda_0)
 \to \mathcal{C}(T(\lambda_0))$ is continuous if and only if
 $\lambda_0$ is of disjoint type.
\end{cor}
We say a $d$-invariant post-critically finite rational lamination
$\lambda_0$ is of \emph{disjoint type} if it is the rational
lamination of a post-critically finite polynomial such that all Fatou
critical points are simple and periodic, and lie in different orbits
(an equivalent definition without polynomial realization is given
in Section~\ref{sec-comb}).
In this case, renormalizations consist of quadratic-like mappings,
so the corresponding straightening map is continuous.

Also, in the same article \cite{Douady-Hubbard-poly-like}, Douady and
Hubbard have already given an example of
cubic-like family whose straightening map is discontinuous.
Their example strongly suggests that straightening maps
are often discontinuous. However, their example is constructed by
putting some invariant complex dilatation outside filled Julia sets of
polynomials, 
hence their argument does not allow us to know whether a given
straightening map is continuous or not.

Epstein \cite{Epstein-manuscript} have also proved that straightening
maps are discontinuous on the boundary of the \emph{main} hyperbolic
component (i.e., the one containing the map hybrid equivalent to the
power map), and his result can be 
generalized to all hyperbolic components such that an attracting
periodic orbit attracts at least two critical points, by use of the
author's result \cite{Inou-semiconj}.
Epstein's result and our result have many similarities; 
both depend on parabolic implosion, and prove existence of an analytic
conjugacy between renormalization and its straightening assuming
that the straightening map is continuous.
However, the proofs for the existence are completely different.
Epstein's proof depends on analytic dependence of Ecalle-Voronin
invariants and our proof depends on combinatorial constructions with the
help of rational laminations. This difference yields completely different
sequences for which the straightening map is discontinuous; 
Epstein's one is in a hyperbolic component, and ours is in the
bifurcation locus. 

In the case of cubic polynomials,
fully renormalizable polynomials are divided into four types in terms of
mapping schema, according to Milnor \cite{Milnor-cubic}; adjacent,
bitransitive (bicritical), capture and disjoint.
The target space of straightening maps are determined by these types:
It is the cubic connectedness locus $\mathcal{C}(3)$ for adjacent type, 
the connectedness locus of the biquadratic family
$\mathcal{C}(2\times2)=\{(a,b) \in \C^2;~K((z^2+a)^2+b)$ is
connected$\}$ for bitransitive type,
the full family of connected quadratic filled Julia sets
$\mathcal{MK}=\{(c,z);\ c \in \mathcal{M},\ z \in K(z^2+c)\}$ for
capture type, and
the product space of the Mandelbrot set with itself $\mathcal{M} \times
\mathcal{M}$ for disjoint type.
Any disjoint type straightening map is continuous because it consists
of straightening maps of two quadratic-like families.
Straightening maps are not continuous for all the other cases.

On the other hand, for capture renormalizations, 
straightening maps are continuous on each fiber.
Buff and Henriksen \cite{Buff-Henriksen} have proved
there are natural quasiconformal embeddings of the filled Julia set
$K(\lambda z+z^2)$ for $|\lambda| \le 1$ into the connectedness locus of
a cubic one-parameter family of the form $\{\lambda z+az^2+z^3\}_{a \in \C}$, 
and we proved that any connected filled
Julia set can be homeomorphically embedded to the connectedness locus of
any higher degree polynomials \cite{Inou-intertwine}. 
Furthermore, we have proved in the preceding paper
\cite{Inou-Kiwi-straightening} that for a cubic 
rational lamination of primitive capture type, 
the straightening map $\chi$ is surjective onto $\mathcal{MK}$,
and its restriction to $\mathcal{K}_c = \chi^{-1}(\{c\}\times K(z^2+c))$
for each $c \in \mathcal{M}$ can be extended to a quasiconformal
embedding, possibly after desingularizing the one-dimensional analytic
set containing $\mathcal{K}_c$.
Therefore, by Main Theorem, such quasiconformal embeddings of connected
Julia sets does not move continuously on polynomials.

Partially, the proof of Main Theorem can be also applied to
renormalizable rational maps and transcendental entire maps. We discuss
that in the last section (Section~\ref{sec-rat-trans}).

Our argument needs two-dimensional bifurcations to prove discontinuity:
One is bifurcation of two critical orbits in one grand orbit, and the
other is parabolic 
bifurcation. Therefore, we cannot apply our argument to a one-parameter
family of polynomials.
Hence it is natural to ask whether we can get discontinuous
straightening maps for smaller parameter spaces.
And we may also ask whether parabolic bifurcation is the unique
possibility to get discontinuity.
So it is natural to ask the following:
\begin{question}
 Can straightening maps be discontinuous under the following conditions?
 \begin{enumerate}
  \item On real polynomial families.
  \item On dynamically defined complex one-parameter spaces.
  \item On anti-holomorphic one-parameter families.
  \item At non-parabolic parameters (having a Siegel disk, or an
	invariant line field on the Julia set).
 \end{enumerate}
\end{question}
The typical example of the third families is the \emph{unicritical
anti-holomorphic family} of degree $d$, which is of the form $\bar{z}^d+c$.
The connectedness locus of this family is called the \emph{multicorn},
and it is called the \emph{tricorn} when $d=2$.
Milnor also observed a subset which looks similar to the tricorn in the
real cubic family \cite{Milnor-cubic}, which is the simplest family of (i).
By numerical experiment, one can see many ``umbilical cords''
accumulating to hyperbolic components, which do not land at a point.
In fact, Mukherjee and the author gave a complete description for
the landingness of umbilical cords for multicorns, and then 
proved the straightening map for any ``multicorn-like set'' in a
multicorn of even degree or the real cubic family is not continuous
\cite{Inou:2014aa-2} \cite{Inou:2016aa}. 
Multicorn-like sets naturally appear in other families with
anti-holomorphic symmetry (such as real families) and anti-holomorphic
families. So it is natural to expect that straightening maps are also
discontinuous for such families.

The proof of Main Theorem consists of several steps.
The first step is to relate the continuity of a straightening map for
an \emph{analytic family of polynomial-like mappings with two marked points}
(abbr.\ AFPL2MP) to the multipliers of repelling periodic orbits
(Theorem~\ref{thm-conti-mult}).
Here we consider a similar situation as the example of
discontinuous straightening map by Douady and Hubbard.
We start with a polynomial-like map having a parabolic periodic point
whose basin contains both of the marked points. 
If the straightening map is continuous in a neighborhood of this map,
and it has nice perturbations described in terms of a given repelling
periodic point, parabolic implosion and Lavaurs map, then the modulus of
the multiplier of the repelling periodic point is preserved by
straightening.
Two marked points will be post-critical points in the application, so that
the continuity of $\chi_{\lambda_0}$ implies the continuity of the
straightening map of the corresponding AFPL2MP.

Secondly, we study parabolic bifurcations to find nice perturbations 
so that we can apply the first step (Section~\ref{sec-para-bif}).

Thirdly, we find a nice parabolic map $f_1$ arbitrarily close to a given
Misiurewicz polynomial $f_0$ for which we have nice perturbations in the
second step (Section~\ref{thm-mis-perturb}).

By gluing these three steps together, 
if the straightening map $\chi$ is continuous in a neighborhood of
$f_0$, we can get a hybrid conjugacy preserving multipliers
between quadratic-like restrictions of some iterates of a
renormalization of $f_1$ and its straightening $P_1=\chi(f_1)$.
Thus they are analytically conjugate by
Sullivan-Prado-Przytycki-Urbanski theorem (Theorem~\ref{thm-sppu}).

Then, applying the results on analytic conjugate polynomial-like
restrictions of polynomials \cite{Inou-semiconj} to get a
contradiction.

Since this proof is constructive, we can get some information at which
parameter a straightening map is not continuous.
See Remark~\ref{rem-discont-at-f_n} for details.

One of the most difficulties in the proof is that we need to perturb
inside the connectedness locus. Moreover, we need to perturb in the
domain of the straightening map $\mathcal{R}(\lambda_0)$. 
To do this, we construct a sequence of rational laminations or critical
portraits which a desired sequence of maps should have,
then we realize them by polynomials.
However, since those combinatorial objects are not complete invariants,
we cannot apply this construction to parabolic polynomials.
Hence we first construct Misiurewicz polynomials and take a limit to
find such perturbations. 
To show that Misiurewicz polynomials constructed in this way
and their limits are in $\mathcal{R}(\lambda_0)$,
we also need some facts saying that $\mathcal{R}(\lambda_0)$ contains
plenty of dynamics (Theorem~\ref{thm-cpt} and Theorem~\ref{thm-pcf-tuning}),
proved in the preceding paper \cite{Inou-Kiwi-straightening}.

\begin{ackn}
 The author would thank Mitsuhiro Shishikura for helpful comments.
 He would also thank Peter Ha\"issinsky, Tomoki Kawahira and Jan Kiwi
 for valuable discussions.
 He would also like to express his gratitude to Institut de
 Math\'ematiques de Toulouse for its hospitality during his visit during
 2007/2008 when this paper was mostly written.
\end{ackn}


\section{Polynomial-like mappings}

In this section, we recall the notion of polynomial-like mappings.
We also describe Sullivan-Prado-Przytycki-Urbanski theorem, which gives
a sufficient condition for given two polynomial-like mappings to be
analytically equivalent (Theorem~\ref{thm-sppu}).
We apply this theorem to polynomial-like restrictions of rational maps,
and prove existence of a global conjugacy in a weak sense
(Theorem~\ref{thm-mult-semiconj}).

\begin{defn}[Polynomial-like mapping]
 A \emph{polynomial-like mapping} is a proper holomorphic map $f:U'
 \to U$ with $U' \Subset U \subset \C$.
 We always assume the degree of $f$ is at least two.
 The \emph{filled Julia set} $K(f)=K(f;U',U)$ is defined by
 \[
  K(f;U',U)=\bigcap_{n \ge 0} f^{-n}(U')
 \]
 and we call $J(f) = J(f;U',U)=\partial K(f;U',U)$ the \emph{Julia set}.
\end{defn}

We introduce the notion of \emph{external markings},
which is necessary to distinguish polynomials whose renormalizations are
hybrid equivalent but combinatorially different.
\begin{defn}[Access and external marking]
 Let $f:U' \to U$ be a polynomial-like mapping.
 A \emph{path to $K(f)$} is a path $\gamma:[0,1] \to U'$ such that
 $\gamma(0) \in J(f)$ and $\gamma((0,1]) \subset U' \setminus K(f)$.
 For a path $\gamma$ to $K(f)$, there exists a unique component of
 $f(\gamma) \cap U'$ which is also a path to $K(f)$ (after a suitable
 reparametrization). We denote it by $f_*\gamma$.

 We say two paths $\gamma_0$, $\gamma_1$ to $K(f)$ are 
\emph{homotopic} if they are homotopic rel $K(f)$, i.e., if there exists a
 homotopy $\gamma:[0,1] \times [0,1] \to U'$ such that
 $\gamma(0,t)=\gamma_0(t)$, $\gamma(1,t)=\gamma_1(t)$ and
 $\gamma(s,0)=\gamma_0(0)$. 
 An \emph{access} for $f:U' \to U$ is a homotopy class of paths to
 $K(f)$.
 We say an access $[\gamma]$ is \emph{invariant} 
 if $f_*\gamma$ is homotopic to $\gamma$.
 It is easy to see that this definition does not
 depend on the choice of representatives.

 An \emph{external marking} of a polynomial-like mapping is an
 invariant access.
 An \emph{externally marked polynomial-like mapping} is 
 a pair $(f:U' \to U, [\gamma])$ of a polynomial-like mapping and an
 external marking of it.
\end{defn}

\begin{exam}
 Let $f$ be a monic centered polynomial of degree $d \ge 2$.
 For sufficiently large $R>0$,
 let $U=\Delta(R)=\{|z|<R\}$ and $U'=f^{-1}(U)$.
 Then $f:U' \to U$ is a polynomial-like mapping of degree $d$.
 If the external ray $R_f(0)$ of angle $0$ does not bifurcate
 (e.g., when $K(f)$ is connected),
 it lands at a fixed point in $J(f)$ and defines an external marking for
 it. 
 We call it the \emph{standard external marking for $f$}.
\end{exam}

Let $\poly(d)$ be the family of monic centered polynomials of degree
$d$ and let $\mathcal{C}(d)$ be its connectedness locus, i.e., 
the set of all $f \in \poly(d)$ such that the filled Julia set $K(f)$ is
connected.

Equipped with the standard external markings,
$\mathcal{C}(d)$ can be considered as the set of affine
conjugacy classes of externally marked polynomials of degree $d$ with
connected Julia sets.

\begin{defn}[Hybrid equivalence]
 We say two polynomial-like mappings $f:U' \to U$ and $g:V' \to V$ are
 \emph{hybrid equivalent} if there exists a quasiconformal
 homeomorphism $\psi:U'' \to V''$ between neighborhoods of the filled
 Julia sets of $f$ and $g$ such that $\psi \circ f = g \circ \psi$ and
 $\bar{\partial}\psi \equiv 0$ a.e.\ on $K(f;U',U)$.

 For externally marked polynomial-like mappings $(f,[\gamma_f])$ and
 $(g,[\gamma_g])$, 
 we say a hybrid conjugacy $\psi$ between $f$ and $g$ 
 \emph{respects external markings}
 if $\psi(\gamma_f)$ is homotopic to $\gamma_g$.
\end{defn}

The following theorem by Douady and Hubbard
\cite{Douady-Hubbard-poly-like} classifies polynomial-like mappings in
the sense of hybrid conjugacy. It also asserts that most dynamical
properties for polynomials also holds for polynomial-like mappings.
We can further add some information on external markings (see
\cite{Inou-Kiwi-straightening}).
\begin{thm}[Straightening theorem]
 \label{thm-DH-straightening}
 Any polynomial-like mapping $f:U' \to U$
 of degree $d$ is hybrid equivalent to some polynomial $g \in \poly(d)$.

 Moreover, if $K(f;U',U)$ is connected and $f:U' \to U$ is externally marked,
 then such a polynomial $g \in \mathcal{C}(d)$ is unique assuming that a
 hybrid conjugacy respects the external markings, where the external
 marking of $g$ is the standard external marking.
\end{thm}

For a periodic point $x \in \C$ of period $n$ for a holomorphic map $f$,
let us denote its multiplier by $\mult_f(x)$, i.e.,
\[
 \mult_f(x) = (f^n)'(x).
\]

\begin{defn}[Hybrid conjugacy preserving multipliers]
 Let $f:U' \to U$ and $g:V' \to V$ be polynomial-like mappings 
 and $\psi:U \to V$ be a hybrid conjugacy.
 We say that \emph{$\psi$ preserves multipliers}
 if for any periodic point $x$ for $f$, we have
 \[
  |\mult_f (x)| = |\mult_g(\psi(x))|.
 \]
\end{defn}
Note that we only assume the conjugacy preserves the moduli of
multipliers by definition, so it might not preserve the arguments.

\begin{defn}[\Semiconj]
 We say two rational maps $f_1$ and $f_2$ are \emph{\semiconj}
 if there exist rational maps $g$, $\psi_1$ and $\psi_2$ such that
 $\psi_i \circ g=f_i \circ \psi_i$.
\end{defn}
In particular, when $f_1$ and $f_2$ are \semiconj\ , they have the
same degree.

The aim of this section is to prove the following:
\begin{thm}
 \label{thm-mult-semiconj}
 Let $f_1$ and $f_2$ be tame rational maps.
 Assume they have polynomial-like restrictions
 $f_i:U_i' \to U_i$, $i=1,2$ which are hybrid conjugate by a conjugacy
 preserving multipliers.
 Then $f_1$ and $f_2$ are \semiconj.
\end{thm}

The following theorem is essentially proved by Prado \cite{Prado} 
based on the idea given by Sullivan \cite{Sullivan-icm},
and its complete proof was given by Przytycki and Urbanski
\cite{Przytycki-Urbanski-tame}. 
\begin{thm}[Sullivan-Prado-Przytycki-Urbanski]
 \label{thm-sppu}
 Suppose that $f:U' \to U$ and $g:V \to V$ are two tame polynomial-like maps. 
 Then the following are equivalent:
 \begin{enumerate}
  \item \label{ppu-item-mult}
	there exists a hybrid conjugacy between $f$ and $g$ preserving
	multipliers.
  \item \label{ppu-item-anal}
	$f$ and $g$ are analytically conjugate, i.e., there exists a
	conformal isomorphism $\varphi:U'' \to V''$ conjugating $f$ and
	$g$, where $U''$ and $V''$ are neighborhoods of $K(f)$ and
	$K(g)$ respectively.
 \end{enumerate}
\end{thm}

\begin{rem}
 Although Przytycki and Urbanski proved the theorem in the case of
 rational maps, the same proof can be applied to polynomial-like mappings.
 In addition, they only proved the existence of a conformal conjugacy 
 defined between some neighborhoods of their Julia sets in
 proving \ref{ppu-item-mult} $\Rightarrow$ \ref{ppu-item-anal}.
 However, since the existence of such a conjugacy implies that they 
 are externally equivalent, they are analytically conjugate near the
 \emph{filled} Julia sets \cite{Douady-Hubbard-poly-like}.
\end{rem}

In this paper, we do not treat with the precise definition of tameness,
which is done in terms of conformal measures, so we do not give it here.
See \cite{Urbanski-NCP} for details (including the next theorem). 
The important fact is the following.

\begin{thm}
 Every polynomial-like mapping with no recurrent critical points in its
 Julia set
 is tame.
\end{thm}

In this paper, we mainly concern with polynomial-like mappings which are
restrictions of (some iterate of) global dynamics.
In this case, we can prove much stronger conclusion by  
the following theorem \cite{Inou-semiconj}.
\begin{thm}
 \label{thm-semiconj}
 Let $f_1$ and $f_2$ be two rational or entire maps.
 Assume they have polynomial-like restrictions
 $f_i:U_i' \to U_i$, $i=1,2$ which are analytically conjugate.
 Then there exist rational or entire maps $g$, $\varphi_1$ and
 $\varphi_2$ such that $\varphi_i \circ g=f_i \circ \varphi_i$ and $g$
 has a polynomial-like restriction $g:V' \to V$ analytically conjugate
 to $f_i:U_i' \to U_i$. 

 Furthermore, if both $f_1$ and $f_2$ are rational (resp.\ polynomials),
 then $g$, $\varphi_1$ and $\varphi_2$ are also
 rational (resp.\ polynomials), i.e., $f_1$ and $f_2$ are \semiconj.
 In particular, $f_1$ and $f_2$ have the same degree.
\end{thm}

Theorem~\ref{thm-mult-semiconj} is an easy consequence of these
theorems.


\section{Analytic families of polynomial-like mappings}

In this section, we briefly review the notion of analytic family of
polynomial-like mappings and its straightening map.
We also consider families with marked points.

\begin{defn}[AFPL]
 An \emph{analytic family of polynomial-like mappings}
 (abbr.\ \emph{AFPL}) of degree $d$ is a family
 $\f = (f_\mu:U_\mu' \to U_\mu)_{\mu \in
 \Lambda}$ of 
 polynomial-like mappings of degree $d$ parameterized by a complex
 manifold $\Lambda$ such that
 \begin{enumerate}
  \item $\mathcal{U}=\{(\mu,z);\ z \in U_\mu\}$ and
	$\mathcal{U}' = \{(\mu,z);\ z \in U_\mu'\}$ are
	homeomorphic over $\Lambda$ to $\Lambda \times \Delta$, where
	$\Delta$ is the unit disk;
  \item the projection from the closure of $\mathcal{U}'$ in
	$\mathcal{U}$ to $\Lambda$ is proper;
  \item the map $f:\mathcal{U}' \to \mathcal{U}$, $f(\mu,z) =
	(\mu,f_\mu(z))$ is holomorphic and proper.
 \end{enumerate}
 Let $\mathcal{C}(\f)=\{\mu \in \Lambda;\ K(f_\mu)$ is
 connected$\}$ be the \emph{connectedness locus} of $\f$. 
\end{defn}

\begin{defn}[External markings for AFPL]
 Let $\f = (f_\mu:U_\mu' \to U_\mu)_{\mu \in
 \Lambda}$ be an AFPL.
 An \emph{external marking for $\f$} is a family of paths
 $([\gamma_\mu])_{\mu \in \mathcal{C}(\f)}$ such that
 $(\mu,t) \mapsto \gamma_{\mu}(t)$ is continuous for $(\mu,t)
 \in \mathcal{C}(\f) \times (0,1]$ and $\gamma_\mu$ is an access for
 $f_\mu$ for each $\mu \in \mathcal{C}(\f)$.

 An \emph{externally marked AFPL} is a pair $(\f,[\gamma_\mu])$
 of an AFPL and an external marking for it.
\end{defn}

Notice that we only consider external markings for maps with connected
Julia sets.

\begin{rem}
 We do not require that $(\mu,t) \mapsto \gamma_\mu(t)$ is
 continuous on $\mathcal{C}(\f) \times [0,1]$.
 Indeed, consider the quadratic family $\bfQ=(Q_c(z)=z^2+c)_{c \in \C}$
 with standard external marking.
 Namely, let $\gamma_c(t) = \phi_c^{-1}(\exp(t))$ where $\phi_c$ is the
 B\"ottcher coordinate for $Q_c$,
 and consider $[\gamma_c]$ as an external marking.
 Then $(c,t) \mapsto \gamma_c(t)$ is not continuous at $(1/4,0)$
 because of the parabolic implosion (discontinuity of the filled Julia
 set), although $c \mapsto \gamma_c(0)$ is still continuous at $c=1/4$
 in this case.
 Thus it is not reasonable to require continuity at $t=0$.

 Moreover, the endpoint $\mu \mapsto \gamma_\mu(0)$ can even be discontinuous.
 For example, consider a family of cubic polynomials of the form
 $f_\mu(z) = z + \frac{1}{2}z(z-1)^2+\mu x$. When
 $\mu=0$, $1$ is a parabolic fixed points whose basin contains both of
 the critical points. 
 The holomorphic fixed point index for $1$ is $2$, hence it is
 parabolic-attracting. 
 In particular, if $\mu>0$ is sufficiently small, the parabolic fixed
 point $1$ splits into two attracting fixed points, hence $f_\mu$ has
 connected Julia set. Moreover, since $f_\mu$ is real, 
 the external ray of angle $0$ is the connected component of $\R
 \setminus K(f_\mu)$ containing sufficiently large real numbers.
 Hence 
 \[
  R_{f_\mu}(0) = 
 \begin{cases}
  (1,\infty) & \mu=0, \\
  (0,\infty) & \mu>0.
 \end{cases}
 \]
 So the landing point of $R_{f_\mu}(0)$ is $0$ for $\mu>0$, but $1$ for $\mu=0$.
\end{rem}

By the straightening theorem, we can naturally define a map from
the connectedness locus $\mathcal{C}(\f)$ for an externally marked AFPL
to $\mathcal{C}(d)$.
\begin{defn}[Straightening maps for AFPL]
 Let $\f=(f_\mu:U_\mu' \to U_\mu)_{\mu \in \Lambda}$ be
 an AFPL and $\Gamma=[\gamma_\mu]_{\mu \in \mathcal{C}(\f)}$ be an
 external marking.
 The \emph{straightening map} $\chi_{\f, \Gamma}:
 \mathcal{C}(\f) \to \mathcal{C}(d)$ is defined as follows:
 $\chi_{\f,\Gamma}(\mu)=g_\mu$ if $f_\mu:U_\mu' \to
 U_\mu$ is hybrid equivalent to $g_\mu$
 respecting the external markings (the external marking for $g_\mu$
 is the standard external marking).
\end{defn}
In the following, whenever we consider an AFPL, we fix an external
marking for each AFPL.
Hence we omit $\Gamma$ for simplicity and write
$\chi_{\f}$ instead of $\chi_{\f,\Gamma}$.

The following theorem is proved by Douady and Hubbard
\cite[Chapter~2, \S 7]{Douady-Hubbard-poly-like}.
\begin{thm}
 \label{thm-quad-conti}
 The straightening map for an AFPL of degree two
 is continuous and can be extended continuously on $\Lambda$.
 \marginpar{Continuous extension necessary?}
\end{thm}

This theorem depends on the following lemma and quasiconformal rigidity 
of quadratic polynomials in the boundary of the Mandelbrot set (see
also Theorem~\ref{thm-conti}).
Namely, for any $f \in \partial \mathcal{C}(2)$, if $g \in
\mathcal{C}(2)$ is quasiconformally conjugate to $f$, then $g=f$.

\begin{lem}[{\cite[Chapter~2, \S 7]{Douady-Hubbard-poly-like}}]
 \label{lem-qc-conj}
 Consider an analytic family $\f = (f_\mu:U_\mu' \to
 U_\mu)_{\mu \in \Lambda}$ of polynomial-like mappings of degree $d$
 and let $\chi_{\f}: \mathcal{C}(\f) \to \mathcal{C}(d)$ be its
 straightening map.
 
 Assume $\mu_n \to \mu$ in $\mathcal{C}(\f)$ and $\chi_{\f}(\mu_n)$
 converges to some $P \in \mathcal{C}(d)$.
 Then there exist $K \ge 1$ independent of $n$ and a
 $K$-quasiconformal hybrid conjugacy $\psi_n$ between $f_{\mu_n}$ and
 $\chi_{\f}(\mu_n)$ such that
 $\psi_n$ converges uniformly to a $K$-quasiconformal conjugacy $\psi$
 between $f_\mu$ and $P$ by passing to a subsequence.
 In particular, $\chi_{\f}(\mu)$ and $P$ are quasiconformally equivalent.
\end{lem}

However, quasiconformal rigidity does not hold for
polynomials of higher degree in the bifurcation locus.
For example, if the basin of a parabolic periodic orbit contains two or
more critical points with distinct grand orbits, then you can deform it
quasiconformally to another polynomial.
Discontinuity of straightening maps is caused by such a lack of
quasiconformal rigidity and Douady and Hubbard used such a parabolic
polynomial to construct an example of discontinuous straightening map
\cite[Chapter 3, \S 4]{Douady-Hubbard-poly-like}.

\begin{defn}[Marked points, AFPL($n$)MP]
 Let $\f_0=(f_\mu:U_\mu' \to U_\mu)$ be an AFPL.
 A \emph{marked point} $x_\mu$ for $\f$ 
 is a holomorphic map $x:\Lambda \to \C$ such that $x_\mu=x(\mu)
 \in U_\mu'$.

 An \emph{analytic family of polynomial-like mappings with a marked
 point} (abbr.\ \emph{AFPLMP}) is a family 
 \[
  \f=(f_\mu:U_\mu' \to U_\mu, x_\mu)_\mu \in
 \Lambda
 \]
 such that $\f_0=(f_\mu:U_\mu' \to
 U_\mu)$ is an AFPL and $x_\mu$ is a marked point for $\f_0$.

 Let us denote 
 \[
  \mathcal{CK}(\f)=\{\mu \in \Lambda;\ K(f_\mu)
  \mbox{ is connected and } x_\mu \in K(f_\mu)\}.
 \]
 The \emph{straightening map} $\chi_\f :\mathcal{CK}(\f) \to
 \mathcal{CK}(d)$, where 
 \[
  \mathcal{CK}(d)=\{(g,z);\ g \in \mathcal{C}(d) \mbox{ and } z \in K(g)\}
 \subset \poly(d) \times \C,
 \]
  is defined as follows.
 Let $\chi_\f(\mu)=(g_\mu,z_\mu)$ when $f_\mu:U_\mu' \to
 U_\mu$ is hybrid equivalent to $g_\mu$ by a hybrid conjugacy
 $\psi$ and $\psi(x_\mu)=z_\mu$ (note that
 $\psi|_{K(f_\mu)}$ is unique under the assumption that $\psi$ respects
 external markings).

 We need also consider an \emph{analytic family of polynomial-like
 mappings with two marked points} (abbr.\ \emph{AFPL2MP}).
 More generally, for $n \ge 1$,
 we say a family
 \[
  \f=(f_\mu:U_\mu' \to U_\mu, x_{1,\mu},\dots,x_{n,\mu})_{\mu \in
 \Lambda}
 \]
 is an \emph{AFPL$n$MP} if $\f_0=(f_\mu)_{\mu \in \Lambda}$ is
 an AFPL and $x_1,\dots,x_n$ are marked points for $\f_0$.
 (equivalently, $\f_k =(f_\mu,x_{k,\mu})_{\mu \in
 \Lambda}$ is AFPLMP for any $k=1,\dots,n$).
 We can similarly define the \emph{straightening map} as follows:
 Let 
 \begin{align*}
  \mathcal{CK}(\f) &=\bigcap_{k=1}^n \mathcal{CK}(\f_k), \\
  \mathcal{CK}^n(d) &= \{(g,z_1,\dots,z_n);\ g \in \mathcal{C}(d),\ 
  z_k \in K(g) \mbox{ for }k=1,\dots,n\},
 \end{align*}
 and define $\chi_\f:\mathcal{CK}(\f) \to \mathcal{CK}^n(d)$ by
 $\chi_\f(\mu)=(g,z_1,\dots,z_n)$ when $\chi_{\f_k}(\mu)=(g,z_k)$.
\end{defn}

We will discuss continuity of straightening maps for AFPL2MP in
Section~\ref{thm-conti-mult}.


\section{Parabolic implosion}

Here we recall the notion of parabolic implosion and geometric limit.
For more details, see \cite{Douady-Julia-set},
\cite{Douady-Sentenac-Zinsmeister}, \cite{Shishikura-HD2} and
\cite{Shishikura-bifurcation}.

Let $f_0$ be a holomorphic map defined near $0$.
Assume $0$ is a non-degenerate $1$-parabolic fixed point, that is,
$f_0(0)=0$, $f_0'(0)=1$ and $f_0''(0)\ne 0$.
By a change of coordinate, we may assume $f_0$ has the form
\[
 f_0(z) = z + z^2 + O(z^3) \quad (z \to 0).
\]
For $\varepsilon>0$, consider two disks
\begin{align*}
 D_{f_0,\attr}
 &= \{z \in \C;\ |z+\varepsilon|<\varepsilon\}, &
 D_{f_0,\rep} 
 &= \{z \in \C;\ |z-\varepsilon|<\varepsilon\}. \\
 \intertext{%
 If $\varepsilon$ is sufficiently small, then
 }
 f(D_{f_0,\attr}) &\subset D_{f_0,\attr}, &
 f(D_{f_0,\attr}) &\supset D_{f_0,\attr}, \\
 \intertext{and there exist conformal maps}
 \Phi_{f_0,\attr}&:D_{f_0,\attr} \to \C, &
 \Phi_{f_0,\rep}&:D_{f_0,\rep} \to \C
\end{align*}
satisfying the Abel equation:
\begin{equation}
 \label{eqn-Fatou-coord}
  \Phi_{f_0,*}(f_0(z))=\Phi_{f_0,*}(z)+1 \quad (*=\attr,\rep),
\end{equation}
where both sides are defined.
We call $\Phi_{f_0,\attr}$ (resp.\ $\Phi_{f_0,\rep}$) an 
\emph{attracting Fatou coordinate} (resp.\ \emph{repelling Fatou
coordinate}) for $f_0$.
They are unique up to post-composition by translation.
If $f$ is a rational map or an entire map, we can extend Fatou
coordinates by the functional equation \eqref{eqn-Fatou-coord}: 
\begin{itemize}
 \item The domain of $\Phi_{f_0,\attr}$ can be extended to the whole
       basin of attraction $B_0$ of $0$.
 \item The domain of $\Psi_{f_0,\rep}=\Phi_{f_0,\rep}^{-1}$ can be
       extended to the whole complex plane $\C$.
\end{itemize}
For $\tau \in \C$, let us define $g_{f_0,\tau}:B_0 \to \C$ by
\[
 g_{f_0,\tau}(z) = \Psi_{f_0,\rep}(\Phi_{f_0,\attr}(z)+\tau).
\]
Then $g_{f_0,\tau}$ commutes with $f_0$, i.e.,
$g_{f_0,\tau} \circ f_0 =f_0 \circ g_{f_0,\tau}$.
We call $g_{f_0,\tau}$ a \emph{Lavaurs map} of $f$ and we call $\tau$ the
\emph{phase}%
\footnote{It is often called the \emph{lifted} phase, but since we do
not need the ``unlifted'' phase (an element of $\C/\Z$), we simply call
it phase here.}
 of $g_{f_0,c}$.

Let $f$ be a small perturbation of $f_0$.
By taking an affine conjugacy,
we may assume $0$ is still a fixed point for $f$.
Let us denote $f'(0)=\exp(2\pi i \alpha)$ with $\alpha$ small.
Here we consider the case $\alpha \ne 0$ and $|\arg \alpha|<\pi/4$ (or
$|\arg (-\alpha)|<\pi/4$).
Let $x$ be the other fixed point for $f$ close to $0$ (bifurcated from
$0$).
Then $x=-2\pi i \alpha(1+o(1))$ as $f \to f_0$.
Let $D_{f,\attr}$ and $D_{f,\rep}$ be the disks of radii $\varepsilon$
whose boundaries pass through $0$ and $x$ such that $D_{f,\attr}$
intersects the negative real axis and $D_{f,\rep}$ intersects the
positive real axis, so that $D_{f,*} \to D_{f_0,*}$ as $f \to f_0$.
Then there exists a conformal map $\Phi_f$ defined on $D_{f,\attr} \cup
D_{f,\rep}$ such that $\Phi_f(f(z))=\Phi_f(z)+1$.
We call $\Phi_f$ a \emph{Fatou coordinate} for $f$. It is also
unique up to post-composition by translation.
Fatou coordinates depend continuously on $f$ if we normalize properly.
More precisely, what we need is the following fact:
If $f_n \to f_0$, then there exist sequences $c_n$ and
$C_n$ of complex numbers such that 
\begin{align*}
 \Phi_{f_n}(z)+c_n &\to \Phi_{f_0,\attr}(z)\mbox{ on }D_{f_0,\attr}, &
 \Phi_{f_n}(z)+C_n &\to \Phi_{f_0,\rep}(z)\mbox{ on }D_{f_0,\rep}.
\end{align*}
as $n \to \infty$. Hence we have for $z \in B_0$
\begin{align*}
 f_n^k(z) &= \Phi_{f_n}^{-1}(\Phi_{f_n}(z)+k) \\
 &= \Phi_{f_n}^{-1}(\Phi_{f_n}(z)+c_n+(k-c_n+C_n)-C_n).
\end{align*}

Now assume $c_n - C_n$ converges in $\C/\Z$, i.e., assume that there exists a
sequence $k_n \in \Z$
such that 
\[
 \lim_{n \to \infty} k_n-c_n+C_n=\tau \in \C.
\]
Then we get a local uniform convergence on $B_0$:
\[
 f_n^{k_n}(z) \to g_{f_0,\tau}(z).
\]
We say that \emph{$f_n$ converges geometrically to
$(f_0,g_{f_0,\tau})$} and denote 
\[
 f_n \xrightarrow{\geom}(f_0,g_{f_0,\tau}).
\]


\section{Continuous straightening maps and multipliers}

The following theorem relates continuity of straightening map to a
condition on multipliers, and is the key to get discontinuity of
straightening maps.
\begin{thm}
 \label{thm-conti-mult}
 Let $\f=(f_\mu:U_\mu' \to U_\mu,
 x_\mu,y_\mu)_{\mu \in \Lambda}$
 be an AFPL2MP of degree $d \ge 2$.
 Assume 
 \begin{enumerate}
  \item for any $\mu \in \Lambda$, $0$ is a fixed point for
	$f_\mu$;
  \item $\alpha_\mu$ is a marked repelling periodic point for $f_\mu$;
  \item $0$ is a non-degenerate $1$-parabolic
	fixed point for $\mu=\mu_0$;
  \item $x_{\mu_0}=y_{\mu_0}$ and $x_{\mu_0}$ lies in the basin of $0$
	for $f_{\mu_0}$;
  \item there exist sequences 
	$\mu_n \xrightarrow{n \to \infty} \mu_0$ and
	$\mu_{n,m} \xrightarrow{m \to \infty} \mu_n$ such that
  \begin{itemize}
   \item $\mu_n, \mu_{n,m} \in \mathcal{CK}(\f)$;
   \item $x_{\mu_n} \ne y_{\mu_n}$ for $n \ge 1$;
   \item $0$ is a non-degenerate $1$-parabolic fixed point for
	 $f_{\mu_n}$;
   \item $f_{\mu_{n,m}} \xrightarrow{\geom}
	 (f_{\mu_n},g_n)$ as $m \to \infty$ for some Lavaurs map
	 $g_n$ such that $g_n(x_{\mu_n})=\alpha_{\mu_n}$.
	 In particular, $0$ is no more a parabolic fixed point for
	 $f_{\mu_{n,m}}$.
   \item $g_n \to g$ for some Lavaurs map $g$ for $f_{\mu_0}$ such
	 that $g'(x_{\mu_0})\ne 0$.
  \end{itemize}
  \item \label{item-conti}
	$\chi_\f(\mu_{n,m}) \to \chi_\f(\mu_n)$
	as $m \to \infty$
	and $\chi_\f(\mu_n) \to \chi_\f(\mu_0)$
	as $n \to \infty$.
 \end{enumerate}
 Then
 \[
 |\mult_{f_{\mu_0}}(\alpha_{\mu_0})| = 
 |\mult_{P_{\mu_0}}(\psi_{\mu_0}(\alpha_{\mu_0}))|,
 \]
 where $\chi_\f(\mu)=(P_\mu,x_\mu^P,y_\mu^P)$ and
 $\psi_\mu$ is a hybrid conjugacy between $f_\mu$ and $P_\mu$.
\end{thm}

Roughly speaking, if the moduli of the multipliers of the corresponding
repelling periodic points $\alpha_{\mu_0}$ and $\alpha_{\mu_0}^P$ are
different and there are plenty of perturbations in $\mathcal{C}(\f)$, 
then the straightening map $\chi_{\f}$ is discontinuous.

The rest of this section is devoted for the proof of this theorem.
We may assume that the hybrid conjugacy
$\psi_{\mu_{n,m}}$ converges to a quasiconformal conjugacy $\varphi_n$
between $f_{\mu_n}$ and $P_{\mu_n}$ as $m \to \infty$ by
Lemma~\ref{lem-qc-conj}.
Then by the continuity \ref{item-conti}, we have
\[
 \varphi_n(x_{\mu_n}) = \lim_{m \to
 \infty}\psi_{\mu_{n,m}}(x_{\mu_{n,m}})
 = \psi_{\mu_n}(x_{\mu_n})=x_{\mu_n}^P,
\]
and similarly we have $\varphi_n(y_{\mu_n})=y_{\mu_n}^P$.

On the other hand, since $f_{\mu_{n,m}} \xrightarrow{\geom}
(f_{\mu_n},g_n)$, there exists a sequence $(k_{n,m})$ such
that 
\[
 f_{\mu_{n,m}}^{k_{n,m}} \to g_n.
\]
This implies that, if $w$ satisfies that
$f_{\mu_{n,m}}^{k_{n,m}}(\psi_{\mu_{n,m}}^{-1}(w))$ lies in the
definition of $\varphi_n$ for sufficiently large $n$, we have
\[
 P_{\mu_{n,m}}^{k_{n,m}}(w)=
 \psi_{\mu_{n,m}} \circ f_{\mu_{n,m}}^{k_{n,m}} \circ
 \psi_{\mu_{n,m}}^{-1} (w) 
 \to g_n^P (w) := \varphi_n \circ g_n \circ \varphi_n^{-1} (w),
\]
namely, $P_{\mu_{n,m}}$ geometrically converges to 
$(P_{\mu_n},g_n^P)$.
Then 
\begin{align*}
 g_n^P \circ \psi_{\mu_n}(x_{\mu_n})
 &= g_n^P \circ \varphi_n (x_{\mu_n}) \\
 &= \varphi_n(g_n(x_{\mu_n})) \\
 &=  \varphi_n(\alpha_{\mu_n})
\end{align*}
is a repelling periodic point for $P_{\mu_n}$.
Let us denote it by $\alpha^P_{\mu_n}$.
Then $\alpha^P_{\mu_n}=\psi_{\mu_n}(\alpha_{\mu_n})$.
In fact, the proof of \cite[p. 302, Lemma~1]{Douady-Hubbard-poly-like}
can be applied to our case to show that $\psi_{\mu_n}=\varphi_n$ on
the Julia set.
Observe that the combinatorial assumption there holds because 
the images of the unique parabolic basin of $0$ by them are the same.

Now let
\[
 \delta_n(w) =
 \log \left| \frac{\varphi_n(w)-\alpha^P_{\mu_n}}
 {w-\alpha_{\mu_n}} \right|.
\]
By passing to a further subsequence, we may assume $\varphi_n$ also
converges as $n \to \infty$.
Then we have the following ``distortion'' property for $\varphi_n$ at
$\alpha_{\mu_n}$:
\begin{lem}
 \label{lem-dist-diverge}
 Let
 $a_n=|\mult_{f_{\mu_n}}(\alpha_{\mu_n})|$ and 
 $b_n=|\mult_{P_{\mu_n}}(\alpha^P_{\mu_n})|$.
 Then
 \begin{equation}
  \label{eqn-delta}
   \delta_n(w) = \frac{\log b_n-\log a_n}
   {\log a_n}\log|w-\alpha_{\mu_n}| + O(1)
 \end{equation}
 as $w \to \alpha_{\mu_n}$ uniformly on $n$.
 In particular, if $|\mult_{f_{\mu_0}}(\alpha_{\mu_0})| \ne
 |\mult_{P_{\mu_0}}(\alpha^P_{\mu_0})|$, 
 then $\delta_n(w)$ diverges as $w \to \alpha_{\mu_n}$ uniformly on
 sufficiently large $n$.
\end{lem}
 
\begin{rem}
 This Lemma is equivalent that 
 the H\"{o}lder exponent of $\varphi_n$ at $\alpha_{\mu_n}$ is equal to
 $\log b_n/\log a_n$, i.e.,
 \[
  |\varphi_n(w)-\varphi_n(\alpha_{\mu_n})| \asymp
  |w-\alpha_{\mu_n}|^{\frac{\log b_n}{\log a_n}}.
 \]
\end{rem}

\begin{proof}
 Take a small circle $S(r,\alpha_{\mu_n})$ centered at
 $\alpha_{\mu_n}$. If the radius $r>0$ is sufficiently small, 
 then the circle and its image
 $f_{\mu_n}^p(S(r,\alpha_{\mu_n}))$ bounds an annulus $A$,
 where $p$ is the period of $\alpha_{\mu_n}$ for $f_{\mu_n}$.
 There exists some constant $C>1$ such that $C^{-1}<|\delta_n(w)|<C$
 for any $w \in \overline{A}$.
 For $w$ close to $\alpha_{\mu_n}$, there exists some $k>0$ 
 such that $f_{\mu_n}^{kp}(w) \in \overline{A}$.
 Then 
 \begin{align*}
  \delta_n(w) 
  &=
  \log \left| \frac{\varphi_n(w)-\varphi_n(\alpha_{\mu_n})}
  {P_{\mu_n}^{kp}(\varphi_n(w))-\varphi_n(\alpha_{\mu_n})} \right|
  + \log \left|
  \frac{\varphi_n(f_{\mu_n}^{kp}(w))-\varphi_n(\alpha_{\mu_n})}
  {f_{\mu_n}^{kp}(w)-\alpha_{\mu_n}} \right|
  + \log \left| \frac{f_{\mu_n}^{kp}(w)-\alpha_{\mu_n}}
  {w-\alpha_{\mu_n}} \right| \\
  &= \eta_1(w)+\delta_n(f_{\mu_n}^{kp}(w))+\eta_2(w).
 \end{align*}
 Since $f_{\mu_n}^p$ and $P_{\mu_n}^p$ are linearizable near
 $\alpha_{\mu_n}$ and $\alpha_{\mu_n}^P=\varphi_n(\alpha_{\mu_n})$
 respectively, it follows that
 $\eta_1(w)= -k\log a_n+O(1)$ and $\eta_2(w)=k\log b_n+O(1)$.
 Therefore, $\delta_n(w)=k(\log b_n - \log a_n) + O(1)$.
 Furthermore, as $w \to \alpha_{\mu_n}$, $k$ tends to infinity,
 so we have \eqref{eqn-delta}.
 These estimates are uniform on $n$ because of the convergence as $n \to
 \infty$.
\end{proof}

Let $w=g_n(y_{\mu_n})$. If $n$ is sufficiently large, $w$ is close
to $\alpha_{\mu_n}$. Thus we have
\begin{align*}
 \delta_n(w)&= 
 \log \left| \frac{\varphi_n(g_n(y_{\mu_n})) 
                   - \varphi_n(g_n(x_{\mu_n}))}
                  {g_n(y_{\mu_n})-g_n(x_{\mu_n})}\right| \\
 &=
  \log \left| \frac{g_n^P \circ \varphi_n(y_{\mu_n}) 
                   - g_n^P \circ \varphi_n(x_{\mu_n})}
                  {y_{\mu_n} - x_{\mu_n}}\right|
 -
  \log \left| \frac{g_n(y_{\mu_n})-g_n(x_{\mu_n})}
                   {y_{\mu_n} - x_{\mu_n}}\right| \\
 &=
  \log \left| \frac{g_n^P \circ \psi_n(y_{\mu_n}) 
                   - g_n^P \circ \psi_n(x_{\mu_n})}
                  {y_{\mu_n} - x_{\mu_n}}\right|
 -
  \log \left| \frac{g_n(y_{\mu_n})-g_n(x_{\mu_n})}
                   {y_{\mu_n} - x_{\mu_n}}\right|.
\end{align*}
Since $x_{\mu_n}$ and $y_{\mu_n}$ lie in the interior of the filled
Julia set, where $\psi_n$ is holomorphic, we have
\[
 \delta_n(w) =
  \log \left(
  \frac{(g_n^P \circ \psi_n)'(x_{\mu_n})}{g_n'(x_{\mu_n})}
  + O(|y_n-x_n|)
  \right).
\]
Since $g_n \to g$ by assumption, 
we may assume that $g_n^P=\varphi_n \circ g_n \circ \varphi_n^{-1}$ also
converges by passing to a subsequence.
This implies that $|\delta_n(w)|$ is bounded uniformly for sufficiently
large $n$.

Therefore, by Lemma~\ref{lem-dist-diverge}, 
this holds only when
$|\mult_{f_{\mu_0}}(\alpha_{\mu_0})| =
|\mult_{P_{\mu_0}}(\psi_{\mu_0}(\alpha_{\mu_0}))|$.
This proves the theorem.
\qed


\section{Combinatorics of dynamics of polynomials}
\label{sec-comb}

We need to find nice perturbations for a given parabolic
polynomials to apply Theorem~\ref{thm-conti-mult} to a family of
renormalizable polynomials.
To do this end, we need some results in the preceding paper
\cite{Inou-Kiwi-straightening} with Kiwi.
One of the most essential tools to construct such perturbations is a 
combinatorial technique which we call \emph{combinatorial tuning},
which is a combinatorial version of the inverse of straightening.
We recall some definitions and results in
\cite{Inou-Kiwi-straightening} in this section.
We also prove some lemmas for later use.

\subsection{Mapping schemata and skew products}
The notion of mapping schema is introduced by Milnor \cite{Milnor-hyp}
to describe the dynamics of hyperbolic polynomials.
Here, we review the notion of mapping schemata and
consider polynomials and polynomial-like mappings over them, which are
simple generalization of usual polynomials and polynomial-like mappings.

\begin{defn}[Mapping schemata]
 A \emph{mapping schema} is a triple $T=(|T|,\sigma,\delta)$ 
 where $|T|$ is a finite set, and $\sigma:|T| \to |T|$ and $\delta:|T|
 \to \N$ are maps such that for any periodic point $v \in |T|$ for $\sigma$,
 we have
 \[
  \prod_{k=0}^{n-1}\delta(\sigma^k(v)) \ge 2,
 \]
 where $n$ is the period of $v$.
 We call $\delta$ the \emph{degree function} of $T$ and
 \[
  \delta(T) = 1 + \sum_{v \in |T|} (\delta(T)-1)
 \]
 the \emph{total degree} of $T$.

 We call $v \in |T|$ is \emph{critical} if $\delta(v) > 1$.
 We say $T$ is \emph{reduced} if all $v \in |T|$ are critical.
 Here we only consider reduced mapping schemata because we can easily
 extract a reduced schema from a given schema by
 taking the first return map \cite{Milnor-hyp}.

 We say $T$ is \emph{nonempty} if $|T| \ne \emptyset$.
\end{defn}

\begin{defn}
 We say \emph{$T$ has a non-trivial critical relation} if
 either
 \begin{itemize}
  \item there exist critical $v, v' \in |T|$ and $n > 0$ such that 
	$v \ne v'$ and $v=\sigma^n(v')$, or
  \item there exists a critical $v \in |T|$ such that $\delta(v)\ge 3$.
 \end{itemize}
 
 Otherwise, we say $T$ is of \emph{disjoint type}.
\end{defn}

A mapping schema $T$ is of disjoint type if and only if there is exactly
$\delta(T)-1$ periodic orbits,
or, equivalently, every critical $v \in |T|$ is periodic and 
\[
 \prod_{n=0}^{p-1} \delta(\sigma^n(v)) = 2
\]
where $p$ is the period of $v$.

An integer $d (\ge 2)$ represents the trivial schema of total degree
$d$, i.e., $d=(\{pt\},id,d)$ (see Figure~\ref{fig-schemata}).
Another important example of mapping schemata is the schema of
capture type:
Let $d \ge 3$ and let $T_{\capt,d}=(|T_{\capt}|, \sigma_{\capt},
\delta_{\capt,d})$ be defined by
\begin{itemize}
 \item $|T_{\capt}|=\{v_1,v_2\}$,
 \item $\sigma_{\capt}(v_j)=v_1$ for $j=1,2$,
 \item $\delta_{\capt,d}(v_1)=2$ and $\delta_{\capt,d}(v_2)=d-1$ for $j=1,2$.
\end{itemize}
For simplicity, we denote by $T_{\capt}$ the degree $3$ capture schema
$T_{\capt,3}$.
\begin{figure}[h]
 \begin{align*}
  \entrymodifiers={++[o][F-]}
  \xymatrix @-1pc {
  1 \ar@(rd,ru)[]_d} 
  & &
  \entrymodifiers={++[o][F-]}
  \xymatrix @-1pc {
  1 \ar@(ld,lu)[]^2 & 2 \ar[l]^{d-1}
  } 
  \end{align*}
 \caption{The trivial schema (left) 
 and the capture schema $T_{\capt,d}$ (right) of total degree $d$.}
 \label{fig-schemata}
\end{figure}
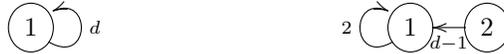

\begin{defn}[Polynomials over mapping schemata and universal polynomial
 model spaces]
 Let $T=(|T|,\sigma,\delta)$ be a mapping schema.
 A \emph{polynomial over $T$} is a map of the form
 \begin{align*}
  f:|T| \times \C &\to |T| \times \C, & f(v,z)&=(\sigma(v),f_v(z))
 \end{align*}
 such that $f_v$
 is a polynomial of degree $\delta(v)$.
 We say $f$ is \emph{monic centered} if $f_v$ is so for all $v \in
 |T|$.
 The \emph{universal polynomial model space} $\poly(T)$ is the set
 of all monic centered polynomials over $T$.

 For a polynomial $f$ over $T$, the \emph{filled Julia set} $K(f)$
 is the set 
 of points whose forward orbit is precompact, and the \emph{Julia
 set} $J(f)$ is the boundary of $K(f)$.
 We say $K(f)$ is \emph{fiberwise connected} if the fiber $K(f,v)
 =\{z \in \C;\ (v,z) \in K(f)\}$ is connected for all $v \in |T|$.
 The (\emph{fiberwise}) \emph{connectedness locus}
 $\mathcal{C}(T)$ is the set of all maps $f \in \poly(T)$ with
 fiberwise connected filled Julia set.
\end{defn}
Observe that a polynomial over the trivial schema $d$ is simply a
polynomial of degree $d$.
Thus the definition of $\poly(d)$ is consistent and we can treat normal
polynomials and polynomials over mapping schemata at the same time.

For $f \in \poly(T)$, $v \in |T|$ and $n>0$, define $f_v^n$ by the equation
\[
 f^n(v,z) = (\sigma^n(v), f_v^n(z)).
\]
Then we have
\[
 K(f) = \{(v,z);\ \{f_v^n(z)\}_{n \ge 0} \mbox{ is bounded}\}.
\]

For a polynomial $f \in \poly(T)$ over a mapping schema
$T=(|T|,\sigma,\delta)$, there exists the B\"{o}ttcher coordinate
$\phi_f$ at $|T| \times \{\infty\}$, i.e., $\phi_f$ is a
holomorphic map defined on a neighborhood of $|T| \times \{\infty\}$
with values in $|T| \times \C$ such that
\begin{itemize}
 \item $\phi_f$ is tangent to the identity at $|T| \times \{\infty\}$;
 \item it has the form $\phi_f(v,z)=(v,\phi_{f,v}(z))$;
 \item it conjugates $f$ to the power map $(v,z) \mapsto (\sigma(v),z^{\delta(v)})$, i.e., 
       \[
	\phi_f(f(v,z))=(\sigma(v), (\phi_{f,v}(z))^{\delta(v)}).
       \]
\end{itemize}
If $f \in \mathcal{C}(T)$, then we can extend $\phi_f$ 
using the dynamics and obtain a univalent map
\[
 \phi_f: (|T| \times \C) \setminus K(f) \to |T| \times (\C \setminus
 \overline{\Delta}),
\]
which we still denote by $\phi_f$.
Hence for $v \in |T|$ and $\theta \in \R/\Z$, we can define the
\emph{external ray} by
\[
 R_f(v,\theta)=\phi_f^{-1}\left(\{(v,r\exp(2\pi i \theta));\ r>1\}\right).
\]
By definition, we have $f(R_f(v,\theta))=R_f(\sigma(v),\delta(v)\theta)$.
Many results on external rays for usual polynomial also hold for
polynomials over mapping schema and proofs are straightforward.
For example, every external ray of a rational angle lands at a repelling
or parabolic eventually periodic point.

\begin{defn}[Polynomial-like mappings over mapping schemata]
 Let $T=(|T|,\sigma,\delta)$ be a mapping schema.
 A \emph{polynomial-like mapping over $T$} is a proper holomorphic
 skew product over $\sigma$ 
 \[
  \begin{matrix}
    g: & U' & \to & U \\
       & (v,z)& \mapsto & (\sigma(v),g_v(z)),
  \end{matrix}
 \]
 such that
 \begin{itemize}
  \item $U' \Subset U$ are subsets of $|T| \times \C$ having the form
	\begin{align*}
	 U' &= \bigcup_{v \in |T|} \{v\} \times U_v',&
	 U  &= \bigcup_{v \in |T|} \{v\} \times U_v,
	\end{align*}
	where $U_v'$ and $U_v$ are topological disks;
  \item $g$ has the form $g(v,z) = (\sigma(v),g_v(z))$ where the degree
	of $g_v:U_v' \to U_{\sigma(v)}$ is equal to $\delta(v)$.
 \end{itemize}
 We may also write it as a collection of proper holomorphic maps 
 \[
  g=(g_v:U_v' \to U_{\sigma(v)})_{v \in |T|}
 \]
 between topological disks in the complex plane.

 The \emph{filled Julia set $K(g)$} is defined as follows:
 \[
  K(g)=\bigcap_{n \ge 0} g^{-n}(U').
 \]
 And the \emph{Julia set} $J(g)$ is the boundary of $K(g)$.
 We say $K(g)$ is \emph{fiberwise connected} if $K(g) \cap \{v\}$ is
 connected for all $v \in |T|$.
 We denote 
 \[
 K(g,v) = \{z \in \C;\ (v,z) \in K(g)\}.
 \]
 By definition, $K(g)$ is fiberwise connected if and only if
 $K(g,v)$ is connected for all $v \in |T|$.
\end{defn}

\begin{defn}[External markings]
 An \emph{external marking} of a polynomial-like mapping $g$ over a
 mapping schema $T=(|T|,\sigma,\delta)$ is a collection of accesses
 $([\gamma_v])_{v \in |T|}$ such that $\gamma_v \subset U_v'$ is a path to $K(g,v)$ and
 $f(\gamma_v) \cap U_{\sigma(v)}' \in [\gamma_{\sigma(v)}]$.
 An \emph{externally marked polynomial-like mapping over $T$} is a
 pair $(g,([\gamma_v]))$ of a polynomial-like mapping over $T$ and 
 an external marking of it.

 Let $f \in \mathcal{C}(T)$. The \emph{standard external marking of
 $f$} is the external marking $([R_f(v,0)])_{v \in |T|}$, defined by the
 external rays of angle zero.
\end{defn}

\begin{defn}[Hybrid equivalence]
 Two polynomial-like mappings $g_1$ and $g_2$ over a mapping schema
 $T=(|T|,\sigma,\delta)$ are \emph{hybrid equivalent}
 if there exists a quasiconformal map $\psi$ defined on a neighborhood
 of $K(g_1)$ such that $\psi \circ g_1 = g_2 \circ \psi$ where both
 sides are defined and $\frac{\partial \psi}{\partial \bar{z}} \equiv 0$
 a.e.\ on $K(g_1)$.

 When $g_1$ and $g_2$ are externally marked, we say that a hybrid conjugacy
 $\psi$ \emph{preserves external markings} if the external marking
 of $g_1$ is mapped to that of $g_2$ by $\psi$.
\end{defn}

We can generalize the straightening theorem
(Theorem~\ref{thm-DH-straightening}) to this case
\cite[Theorem~A]{Inou-Kiwi-straightening}.
\begin{thm}[Straightening theorem for polynomial-like mappings over
 mapping schemata]
 A polynomial-like mapping $g$ over a mapping schema
 $T=(|T|,\sigma,\delta)$ is hybrid equivalent to some $f \in
 \poly(T)$. Furthermore, if $K(g)$ is fiberwise connected and $g$ is
 externally marked, then there exists a unique $f \in \mathcal{C}(d)$ 
 hybrid equivalent to $g$
 such that a hybrid conjugacy between $f$ and $g$ preserves external
 markings, where the external marking of $f$ is the standard external
 marking.
\end{thm}

\subsection{Rational laminations}

Here we recall the notion of \emph{rational lamination},
which is a fundamental tool to discuss combinatorics of polynomials 
with connected Julia sets.
We also consider rational laminations over mapping schemata,
which is necessary to define \emph{combinatorial tuning},
which is the inverse operation of straightening in a combinatorial sense.

For an integer $d>1$, let $m_d$ denote the $d$-fold covering
$\theta \mapsto d\theta$ defined on $\R/\Z$ to itself.
For a mapping schema $T=(|T|,\sigma,\delta)$, 
define a map $m_T:|T| \times \R/\Z \to |T| \times \R/\Z$ 
by 
\[
 m_T(v,\theta) = (\sigma(v),m_{\delta(v)}(\theta))=(\sigma(v),\delta(v)\theta).
\]

We say two sets $A,B \subset \R/\Z$ (or $\{v\} \times \R/\Z$)
are \emph{unlinked} if $B$ is contained in a component of
$\R/\Z \setminus A$.
Note that it is equivalent that $A$ is contained in a component of
$\R/\Z \setminus B$, or that the Euclidean (or hyperbolic) convex hulls
of $\exp(2\pi i A)$ and $\exp(2\pi i B)$ in $\C$ are disjoint.

Let $A \subset \R/\Z$. We say a map $f:A \to f(A) \subset \R/\Z$ is
\emph{consecutive preserving} if 
for any component $(\theta,\theta')$ of $\R/\Z \setminus A$, 
$(f(\theta), f(\theta'))$ is a component of $\R/\Z \setminus
f(A)$. 

\begin{defn}[Invariant rational laminations]
 Let $T=(|T|,\sigma,\delta)$ be a mapping schema.
 An equivalence relation $\lambda$ on $|T| \times \Q/\Z$
 is called a \emph{$T$-invariant rational lamination}
 or a \emph{rational lamination over $T$} if the
 following conditions hold:
 \begin{enumerate}
  \item Each equivalence class is contained in $\{v\} \times \Q/\Z$ for
	some $v \in |T|$.
  \item $\lambda$ is closed in $(|T| \times \Q/\Z)^2$.
  \item Every equivalence class is finite.
  \item Equivalence classes are pairwise unlinked.
  \item For a $\lambda$-equivalence class $A$,
	$m_T(A)$ is also a $\lambda$-equivalence class.
  \item $m_T:A \to m_T(A)$ is consecutive preserving.
 \end{enumerate}
 Let us denote by $\supp(\lambda) \subset |T| \times \Q/\Z$ 
 the union of all non-trivial $\lambda$-classes.

 We may denote $\lambda$ as a collection $(\lambda_v)_{v \in |T|}$
 of (non-invariant) rational laminations on $\Q/\Z$, i.e.,
 $(v,\theta)$ and $(v,\theta')$ are $\lambda$-equivalent if and only if
 $\theta$ and $\theta'$ are $\lambda_v$-equivalent.
\end{defn}

\begin{exam}
 For $f \in \mathcal{C}(T)$, the \emph{rational lamination
 $\lambda_f$} of $f$ is the landing relation of external rays of rational
 angles. 
 Namely, $(v,\theta)$ and $(v,\theta')$ are $\lambda_f$-equivalent if and only if
 the external rays $R_f(v,\theta)$ and $R_f(v,\theta')$ land at the same point.
 By the theorem of Kiwi \cite{Kiwi-ratlamin},
 an equivalence relation $\lambda$ on $|T| \times \Q/\Z$ is a
 $T$-invariant rational lamination if and only if $\lambda$ is the
 rational lamination of some $f \in \mathcal{C}(T)$.
\end{exam}

\begin{defn}[Combinatorial renormalization]
 We say a $T$-invariant rational lamination $\lambda$ is 
 \emph{admissible for $f \in \mathcal{C}(T)$} if $\lambda \subset
 \lambda_f$,
 that is, $R_f(v,\theta)$ and $R_f(v,\theta')$ land at the same point 
 when $(v,\theta)$ and $(v,\theta')$ are $\lambda$-equivalent.
 We also say that $f$ is \emph{$\lambda$-combinatorially
 renormalizable} or \emph{$f$ admits $\lambda$}.
 Let
 \[
 \mathcal{C}(\lambda)=\{f \in \mathcal{C}(T);\ \lambda \subset \lambda_f\}
 \]
 be the set of all polynomials which admit $\lambda$.
\end{defn}

A rational lamination $\lambda$ naturally induces the \emph{unlinked
relation} for irrational angles \cite{Kiwi-ratlamin}, which is closely related to 
\emph{gaps} introduced by Thurston \cite{Thurston}.
\begin{defn}[Unlinked classes]
 Let $T$ be a mapping schema and 
 let $\lambda$ be a $T$-invariant rational lamination.
 We say $(v,\theta),~(v',\theta') \in |T| \times (\R \setminus \Q)/\Z$
 are \emph{$\lambda$-unlinked} 
 if $v = v'$ and for any $\lambda$-equivalence class $\{v\} \times A$, 
 $\theta$ and $\theta'$ lie in the same component of $\R/\Z \setminus A$.
\end{defn}
Observe that $\lambda$-unlinked relation is an equivalence relation and
each equivalence class (\emph{$\lambda$-unlinked class}) is
contained in $\{v\} \times (\R \setminus \Q)/\Z$ for some $v \in |T|$.
A set $\{v\} \times L$ is an unlinked class if and only if $L$ is a
$\lambda_v$-unlinked class.

\begin{lem}
 \label{lem-linked}
 Let $\lambda_0$ and $\lambda$ be rational laminations
 and assume $\lambda \supset \lambda_0$.
 If a $\lambda$-equivalence class $A$ is not a $\lambda_0$-class,
 then there exists some $\lambda_0$-unlinked class $L$ such that 
 $A$ and $L$ are linked (not unlinked) and $\partial L \cap A$ is
 nonempty.
\end{lem}

\begin{proof}
 Let $B \subset A$ be a $\lambda_0$-class.
 Take a component $(s,t)$ of $\R/\Z \setminus B$ which intersects $A$.
 Consider a set
 \[
  F=[s,t] \setminus \bigcup_{\theta,\theta'} [\theta,\theta'],
 \]
 where the union is taken for all $\lambda_0$-equivalent pairs $\theta,
 \theta'$ such that $[\theta,\theta'] \subset (s,t)$.
 Since each pair of such intervals are either disjoint or one contains the
 other, $F$ is a Cantor set removing countably many points. In
 particular, $F$ is uncountable and contained in the derived set of $F$
 itself. Hence $L = F \cap (\R \setminus \Q)/\Z$ is nonempty.
 Furthermore, since $L$ is unlinked with any $\lambda_0$-equivalence
 class, $L$ is in fact a $\lambda_0$-unlinked class.

 Let $B' \subset A \cap (s,t)$ be another $\lambda_0$-class.
 Then since $B$ and $B'$ lie in different components of $L$,
 $A$ and $L$ are linked.

 By construction, $s,t \in A$ lie in the closure of $L$.
\end{proof}

The external rays for $f$ of $\lambda$-equivalent angles cut the phase
space into sectors. This allows us to associate each $\lambda$-unlinked
class with a continuum (compare \cite{Schleicher-fiber}). 
\begin{defn}[Sectors and fibers]
 Let $\lambda$ be a rational lamination over a mapping schema
 $T$ and let $f \in \mathcal{C}(\lambda)$.
 For a $\lambda$-unlinked class $L \subset \{v\} \times \R/\Z$
 and $\lambda$-equivalent angles $(v,\theta), (v,\theta')$, let
 \[
  \sector_f(v,\theta,\theta';L)
 \]
 be the connected component of 
 \[
  (\{v\} \times \C) \setminus \overline{(R_f(v,\theta) \cup
  R_f(v,\theta'))} 
 \]
 containing the external ray $R_f(v,t)$ for every $(v,t) \in L$.
 The \emph{fiber of $L$ for $f$} is defined by:
 \[
  K_f(L) = K(f) \cap \bigcap_{\theta \sim_{\lambda_v} \theta',\ 
 \theta \ne \theta'} \overline{\sector(v,\theta,\theta';L)}.
 \]
\end{defn}

The following proposition (see
\cite[Proposition~3.7]{Inou-Kiwi-straightening}) describes
some basic properties for $\lambda$-unlinked classes and
corresponding fibers.
\begin{prop}
 \label{prop-unlinked-class}
 Let $T$ be a mapping schema and 
 let $\lambda$ be a $T$-invariant rational lamination.
 For $f \in \mathcal{C}(\lambda)$ and a $\lambda$-unlinked
 class $L$, we have the following:
 \begin{enumerate}
  \item $m_T(L)$ is also a $\lambda$-unlinked class.
  \item $f(K_f(L)) = K_f(m_d(L))$.
  \item If $L$ is finite, then
	\begin{enumerate}
	 \item 
	       $m_T:L \to m_T(L)$ is a $\delta(L)$-to-one consecutive
	       preserving map for some $\delta(L)>0$.
	 \item $f:K_f(L) \to K_f(m_T(L))$ has degree $\delta(L)$, i.e., 
	       every point in $K_f(m_T(L))$ has $\delta(L)$ preimages in
	       $K_f(L)$ counted with multiplicity.
	\end{enumerate}
  \item If $L$ is infinite, then
	\begin{enumerate}
	 \item $L$ is eventually periodic by $m_T$.
	 \item There exists a homeomorphism $\alpha_L:\overline{L}/\lambda \to
	       \R/\Z$ such that $\alpha_{m_T(L)} \circ m_T \circ
	       \alpha_L^{-1}$ 
	       is well-defined and coincides with $m_{\delta(L)}$ for some
	       $\delta(L) \ge 1$.
	\end{enumerate}
 \end{enumerate}
\end{prop}

\begin{rem}
 We stated this proposition only for the case of $d$-invariant rational
 laminations in \cite{Inou-Kiwi-straightening}, but the same proof can
 be applied also to rational laminations over mapping schemata.
 
 Similarly, some theorems below in this section are also stated for
 rational laminations and polynomials over mapping schemata, 
 but the proofs are exactly the same (although notations will become more
 complicated), and some of them are immediate consequence 
 from the same result for the usual case.
\end{rem}

\begin{defn}[Critical elements]
 Let $\lambda$ be a $T$-invariant rational lamination.
 For a $\lambda$-class $A$, let $\delta(A)$ denotes the degree of $m_T:A
 \to m_T(A)$. It is well-defined by the consecutive preservingness.
 We say $A$ is \emph{critical} if $\delta(A)>1$.
 Similarly, for $\lambda$-unlinked class $L$, 
 we say $L$ is \emph{critical} if $\delta(L)>1$, where
 $\delta(L)$ is the one defined in Proposition~\ref{prop-unlinked-class}.

 Let $\crit^P(\lambda)$, $\crit^W(\lambda)$ and $\crit^F(\lambda)$ 
 be the set of all critical $\lambda$-classes, finite $\lambda$-unlinked
 classes and infinite $\lambda$-unlinked classes respectively,
 and let $\crit(\lambda)= \crit^P(\lambda) \cup \crit^W(\lambda) \cup
 \crit^F(\lambda)$. 
 We call an element in $\crit^P(\lambda) \cup \crit^W(\lambda)$ (resp.\
 $\crit^F(\lambda)$)
 a \emph{Julia critical element},
 (resp.\ a \emph{Fatou critical element}).
 For a Julia critical element $A$, we say $A$ is \emph{preperiodic}
 if $A \in \crit^P(\lambda)$ and $A$ is \emph{wandering} if $A \in
 \crit^W(\lambda)$. 
\end{defn}
Roughly speaking, critical elements correspond to critical points for
$f \in \mathcal{C}(\lambda)$.
It follows that 
\[
 \delta(v) - 1 = \sum_{A \in \crit(\lambda),~A \subset \{v\} \times \R/\Z}
 (\delta(A)-1).
\]
For $*=P, W, F$, let 
\begin{align*}
 PC^*(\lambda)&=\{m_T^n(w);~w \in \crit^*(\lambda),~n>0\}, &
 PC(\lambda)&= PC^P(\lambda) \cup PC^W(\lambda) \cup PC^F(\lambda), \\
 CO^*(\lambda)&=PC^*(\lambda) \cap \crit(\lambda), &
 CO(\lambda)&= CO^P(\lambda) \cup CO^W(\lambda) \cup CO^F(\lambda).
\end{align*}

\begin{defn}[Post-critically finite, hyperbolic and Misiurewicz laminations]
 We say a $T$-invariant rational lamination $\lambda$ is
 \begin{itemize}
  \item \emph{post-critically finite} if there is no wandering
	critical Julia element (i.e., $\crit^W(\lambda)=\emptyset$),
  \item \emph{hyperbolic} if there is no Julia critical elements
	(i.e., $\crit(\lambda)= \crit^F(\lambda)$), and
  \item \emph{Misiurewicz} if there is no critical
	$\lambda$-unlinked class 
	(i.e., $\crit(\lambda) = \crit^P(\lambda)$).
 \end{itemize}
\end{defn}
Observe that $\lambda$ is post-critically finite if and only if
$PC(\lambda)$ is finite.

\begin{defn}
 A post-critically finite rational lamination $\lambda$ over a mapping
 schema $T$ is \emph{primitive}
 if for any infinite $\lambda$-unlinked classes $w,w' \subset \{v\} \times
 \R/\Z$, there is no $\lambda$-class $A$ such that both $A \cap
 \overline{w}$ and $A \cap \overline{w'}$ are nonempty.
\end{defn}

Even if $\lambda$ is not primitive, such intersections exist essentially
only finitely many:
\begin{lem}
 \label{lem-non-primitive}
 Let $\lambda$ be a post-critically finite rational lamination over a
 mapping schema $T$.
 
 Let $A$ be a $\lambda$-class such that there exist infinite
 $\lambda$-unlinked classes $L_1 \ne L_2$ whose closures intersect $A$.
 Then there exists some $n \ge 0$ such that either 
 \begin{itemize}
  \item $m_T^n(A) \in \crit^P(\lambda)$, or
  \item $m_T^n(L_1) \ne m_T^n(L_2)$ and both lie in $CO^F(\lambda)$.
 \end{itemize}
 In particular, the set of all eventual periods of such
 $\lambda$-classes $A$ is finite.
\end{lem}

One can prove this lemma in a purely combinatorial way; but it is easier
to use a polynomial realization of a given rational lamination by Kiwi
\cite{Kiwi-ratlamin}:
\begin{thm}
 For a given post-critically finite $d$-invariant rational lamination
 $\lambda$, there exists a post-critically finite polynomial $f$ of
 degree $d$ such that $\lambda_f=\lambda$.
\end{thm}
The existence is essentially proved by \cite{Poirier-pcf2},
and implicitly stated in \cite[Section~6-7]{Kiwi-ratlamin} 
(see also \cite[Theorem~5.17]{Inou-Kiwi-straightening}).
The uniqueness is proved in \cite[Theorem~5.18]{Inou-Kiwi-straightening}.

\begin{proof}[Proof of Lemma~\ref{lem-non-primitive}]
 If $m_T^n(A)$ is not critical for any $n \ge 0$,
 then $m_T^n(L_1) \ne m_T^n(L_2)$ for any $n \ge 0$.
 
 By taking the polynomial realization, 
 it is easy to see that all infinite $\lambda$-unlinked classes
 are eventually periodic, and periodic $\lambda$-unlinked classes lie in
 $CO^F(\lambda)$, which is finite.
 Therefore, the lemma follows.
\end{proof}

\begin{defn}[Mapping schemata of rational laminations]
 Let $\lambda$ be a $T_0$-invariant rational lamination such that
 $\crit^F(\lambda)$ is nonempty.
 Define a (reduced) mapping schema
 $T(\lambda)=(|T(\lambda)|,\sigma_{\lambda},\delta_{\lambda})$ by 
 \begin{align*}
  |T(\lambda)| &= \crit^F(\lambda), &
 \sigma_{\lambda}(w) &= m_{T_0}^{\ell_w}(w),
 \end{align*}
 and $\delta_{\lambda}=\delta$ is the one defined in
 Proposition~\ref{prop-unlinked-class},
 where $\ell_w>0$ is the smallest number such that $m_T^{\ell_w}(w) \in
 \crit^F(\lambda)$. 

 We say $\lambda$ has a \emph{non-trivial Fatou critical relation}
 if $T(\lambda)$ has a non-trivial critical relation.
 Otherwise, we say $\lambda$ is of \emph{disjoint type} (it is
 equivalent that $T(\lambda)$ is of disjoint type).
\end{defn}

Proposition~\ref{prop-unlinked-class} guarantees the existence of an
\emph{internal angle system}, which is needed to make straightening maps
well-defined:
\begin{defn}[Internal angle systems]
 Let $T_0=(|T_0|,\sigma_0,\delta_0)$ be a mapping schema and
 let $\lambda$ be a $T_0$-invariant rational lamination with
 $\crit^F(\lambda) \ne \emptyset$.

 An \emph{internal angle system} of $\lambda$ is a collection of maps
 $\alpha=(\alpha_w:\overline{w} \to \R/\Z)_{w \in |T(\lambda)|}$ such
 that $\alpha_w$ induces a homeomorphism between $\overline{w}/\lambda$
 and $\R/\Z$ and 
 \begin{equation}
  \label{eqn-alpha}
   \alpha_{\sigma(w)}\circ m_{T_0}^{\ell_w}(v,\theta) = 
   m_{\delta(w)}(\alpha_w(v,\theta)),
 \end{equation}
 for $(v,\theta) \in \overline{w}$.
\end{defn}

We sometimes omit $v$ and simply write $\alpha_w(\theta)$,
for $\alpha_w$ is defined on $\overline{w} \subset \{v\}\times
\R/\Z$, so $v$ depends only on $w$.
The map $\alpha$ above can also be considered as follows:
\begin{align*}
 \alpha: \bigsqcup_{w \in |T(\lambda)|} \overline{w}
 &\longrightarrow |T(\lambda)| \times \R/\Z, \\
 \overline{w} \ni (v,\theta) 
 &\longmapsto \alpha(v,\theta) := (w,\alpha_w(v,\theta)).
\end{align*}
In this expression, we have the following equality for $(v,\theta) \in
\overline{w}$:
\[
 \alpha \circ m_{T_0}^{\ell_w}(v,\theta)
 = m_{T(\lambda)} \circ \alpha (v,\theta).
\]
Here readers should notice that we need to take the disjoint union
because $\overline{w}$ and $\overline{w'}$ might intersect for $w \ne w'
\in |T(\lambda)|$.

\begin{lem}
 \label{lem-same-period}
 Let $\lambda$ be a rational lamination over a mapping schema
 $T_0=(|T_0|,\sigma_0,\delta_0)$.
 Assume $\crit^F(\lambda) \ne \emptyset$ and 
 let $\alpha=(\alpha_w)_{w \in |T(\lambda)|}$ be an internal angle system.
 Let $(v,\theta) \in \overline{w}$ be periodic of period $p$ by $m_{T_0}$.
 Then $\alpha(v,\theta)=(w,\alpha_w(v,\theta))$ is also periodic of
 period $p'$ by $m_{T(\lambda)}$, where $p'$ is defined by
 \[
  \sum_{n=0}^{p'-1} \ell_{\sigma_\lambda^n(w)}=p.
 \]
 
 In particular, $p$ is not less than the period of $w$ by $\sigma$.
\end{lem}
\marginpar{Need only ``In particular'' part (inequality is enough).}

\begin{proof}
 There exists a sequence $\theta_n$ such that $(v,\theta_n) \in w$ and
 either
 $\theta_n \nearrow \theta$ or $\theta_n \searrow \theta$.
 We may assume $\theta_n \nearrow \theta$ (the other case is similar).
 Then $m_{T_0}^p(v,\theta_n)\nearrow m_{T_0}^p(v,\theta)=(v,\theta)$ 
 and $m_{T_0}^p(v,\theta_n)$ lie in a $\lambda$-unlinked class
 $w'=\sigma_0^p(w)$.
 Hence it follows that $w'=w$ (since otherwise $w$ and $w'$ are linked)
 and we have
 \begin{align*}
  m_{T(\lambda)}^{p'}(w,\alpha_w(v,\theta)) 
  &= \lim_{n \to \infty} m_{T(\lambda)}^{p'}(w,\alpha_w(v,\theta_n)) \\
  &= \lim_{n \to \infty} (\sigma_\lambda^{p'}(w), 
  m_{\delta(\sigma_{\lambda}^{p'-1}(w))}
  \circ \cdots \circ m_{\delta(w)} \circ \alpha_w(v,\theta_n)) \\
  &= \lim_{n \to \infty} (\sigma_0^p(w), \alpha_{\sigma_0^p(w)} \circ
  m_{T_0}^p(v,\theta_n)) \\
  &= (w,\alpha_w(v,\theta))
 \end{align*}
 by the equation~\eqref{eqn-alpha}.

  On the other hand, let $p_1'$ be the period of $(w,\alpha_w(v,\theta))$
 by $m_{T(\lambda)}$. 
 Then $m_{T_0}^{p_1}(v,\theta) \sim_\lambda (v,\theta)$,
 where $p_1=\sum_{n=0}^{p_1'-1} \ell_{\sigma_{\lambda}^n(w)}$.
 Moreover, since
 \[
 \lim_{n \to \infty} m_{T(\lambda)}^{p'_1}(w,\alpha_w(v,\theta_n))
 = m_{T(\lambda)}^{p'_1}(w,\alpha_w(v,\theta))
 = (w,\alpha_w(v,\theta)),
 \]
 it follows that $m_{T_0}^{p_1}(v,\theta_n) \in
 \sigma_\lambda^{p_1'}(w)=w$.
 Therefore, $w$ approaches $m_{T_0}^{p_1'}(v,\theta)$ from the left.
 In the $\lambda$-equivalence class of $(v,\theta)$, this property holds
 only for $(v,\theta)$, thus $m_{T_0}^{p_1}(v,\theta)=(v,\theta)$.
\end{proof}

\subsection{Renormalizations}

\begin{defn}[Renormalizations]
 Let $T_0$ be a mapping schema and $\lambda_0$ be a $T_0$-invariant
 rational lamination with $\crit^F(\lambda_0) \ne \emptyset$.
 We say $f \in \mathcal{C}(\lambda_0)$ is 
\emph{$\lambda_0$-renormalizable} 
 if there exist topological disks $U_w' \Subset U_w$ for each $w \in
 |T(\lambda_0)|$ such that
 \begin{itemize}
  \item $g=(f^{\ell_w}:U_w' \to U_{\sigma(w)})$ is a polynomial-like map
	over $T(\lambda_0)$ with fiberwise connected Julia set.
  \item $K(g,w)=K_f(w)$ for all $w \in |T|$.
 \end{itemize}
 We call $g$ a \emph{$\lambda_0$-renormalization of $f$}.
\end{defn}

\begin{defn}[Straightening map $\chi_{\lambda_0}$]
 Let $T_0$ be a mapping schema and $\lambda_0$ be a $T_0$-invariant
 rational lamination. Let $(\alpha_w:\overline{w} \to \R/\Z)_{w \in
 |T(\lambda_0)|}$ be an internal angle system.
 For each $w \in |T(\lambda_0)|$, take $\theta_w \in \overline{w}$ such
 that $\alpha_w(\theta_w)=0$.
 
 For $f \in \mathcal{R}(\lambda_0)$, define $\chi_{\lambda_0}(f) \in
 \poly(T(\lambda_0))$ as follows;
 let $g$ be a $\lambda_0$-renormalization of $f$.
 Then external rays $([R_f(w,\theta_w)])_{w \in |T(\lambda_0)|}$ defines
 an external marking of $g$.
 Let $\chi_{\lambda_0}(f)$ be the polynomial over $T(\lambda_0)$ hybrid
 equivalent to $g$ preserving external markings.

 This gives a well-defined map
 $\chi_{\lambda_0}:\mathcal{R}(\lambda_0) \to \mathcal{C}(T(\lambda_0))$.
\end{defn}

We recall some results in \cite{Inou-Kiwi-straightening}:
\begin{thm}[Injectivity of straightening maps]
 \label{thm-inj}
 Let $T_0$ be a mapping schema.
 For a post-critically finite $T_0$-invariant rational lamination
 $\lambda_0$, the straightening map $\chi_{\lambda_0}$ is injective.
\end{thm}

We can also give some equivalent conditions about the domain
$\mathcal{R}(\lambda_0)$ of the straightening map.
\begin{prop}
 \label{prop-nonempty}
 Let $\lambda_0$ be a post-critically finite $d$-invariant rational
 lamination having nonempty mapping schema $T(\lambda_0)$. 
 Then the following are equivalent:
 \begin{enumerate}
  \item $\mathcal{R}(\lambda_0)$ is nonempty.
  \item If $A$ is a critical $\lambda_0$-class and $L \in
	CO^F(\lambda_0)$, then $A \cap \overline{L}=\emptyset$.
 \end{enumerate}
\end{prop}

\begin{thm}[Compactness of the renormalizable set for primitive combinatorics]
 \label{thm-cpt}
 Assume a post-critically finite $T_0$-invariant rational lamination
 $\lambda_0$ has nonempty mapping schema $T(\lambda_0)$.
 If $\lambda_0$ is primitive, then 
 \begin{enumerate}
  \item $\mathcal{C}(\lambda_0)=\mathcal{R}(\lambda_0)$, and
  \item $\mathcal{R}(\lambda_0)$ is compact and nonempty.
 \end{enumerate}
\end{thm}

To prove a given $f \in \mathcal{C}(\lambda_0)$ is
$\lambda_0$-renormalizable, we need the following lemma
\cite[Lemma~5.13, Corollary~5.16]{Inou-Kiwi-straightening}, 
which is based on the idea of ``thickening puzzles''
by Milnor \cite{Milnor-lc}.
\begin{lem}
 \label{lem-thickening}
 Let $\lambda_0$ be a rational lamination over a mapping schema $T_0$.
 Then there exists a proper algebraic set $X \subset \poly(T_0)$ such
 that
 $\mathcal{R}(\lambda_0) \supset \mathcal{C}(\lambda_0) \setminus X$.
 Moreover $X$ does not contain $\mathcal{C}(\lambda_0)$ if
 $\mathcal{R}(\lambda_0)$ is nonempty.
 
 More precisely, there exists a finite set of angles $E=E(\lambda_0)
 \subset \supp(\lambda_0)$ such that $f \in \mathcal{C}(\lambda_0)$
 is $\lambda_0$-renormalizable if the landing point of
 $R_f(v,\theta)$ is neither parabolic nor critical for any $(v,\theta)
 \in E$.
\end{lem}
Note that $\mathcal{C}(\lambda_0)$ might be contained in $X$,
and also notice that $\lambda_0$ is primitive if and only if $E$ (hence
$X$) can be taken as the empty set. 

\subsection{Combinatorial tuning}

Let $T_0=(|T_0|,\sigma_0,\delta_0)$ be a mapping schema.
Let $\lambda_0$ be a $T_0$-invariant rational
lamination and fix an internal angle system 
$(\alpha_w:\overline{w} \to \R/\Z)_{w \in |T(\lambda_0)|}$.

\begin{defn}[Combinatorial straightening]
 Let $\lambda$ be a $T_0$-invariant rational lamination containing
 $\lambda_0$.
 the \emph{combinatorial straightening of $\lambda$ with respect to
 $\lambda_0$} is a $T(\lambda_0)$-invariant rational lamination
 $\lambda'=(\lambda'_w)_{w \in |T(\lambda_0)|}$ such that
 $(w,\theta) \sim_{\lambda'_w}(w,\theta')$ if and only if there exist
 $t \in \alpha_w^{-1}(\theta)$ and $s \in \alpha_w^{-1}(\theta')$ such
 that $t$ and $s$ are $\lambda$-equivalent.
\end{defn}

Combinatorial tuning is the inverse operation of combinatorial straightening.
\begin{thm}[Combinatorial tuning]
 \label{thm-comb-tuning}
 Let $\lambda_0$ be a $T_0$-invariant rational lamination and let
 $\lambda'$ be a $T(\lambda_0)$-invariant rational lamination.
 Then there exists a $T_0$-invariant rational lamination $\lambda
 \supset \lambda_0$
 such that the combinatorial straightening of $\lambda$ with respect to
 $\lambda_0$ is $\lambda'$.
 Moreover,
 \begin{enumerate}
  \item if $\lambda_0$ and $\lambda'$ are hyperbolic, then $\lambda$ is
	hyperbolic.
  \item If $\lambda_0$ and $\lambda'$ are post-critically finite, then
	$\lambda$ is post-critically finite.
  \item If $\lambda_0$ is post-critically finite and $\lambda'$ is
	Misiurewicz, then $\lambda$ is Misiurewicz.
  \item If the rational lamination of $f \in \mathcal{R}(\lambda_0)$ is
	$\lambda$, then the rational lamination of $\chi_{\lambda_0}(f)$
	is $\lambda'$.
 \end{enumerate}
\end{thm}

\begin{proof}
 See \cite[Proposition~5.6]{Inou-Kiwi-straightening} for the first two
 statements. The third and fourth statements follow easily.
\end{proof}

By using the combinatorial tuning, we can actually do ``tuning'' in most
cases of post-critically finite dynamics
\cite[Theorem~5.2]{Inou-Kiwi-straightening}:
\begin{thm}[Post-critically finite tuning]
 \label{thm-pcf-tuning}
 Let $\lambda_0$ be a rational lamination over a mapping schema $T_0$
 such that $\mathcal{R}(\lambda_0) \ne \emptyset$.

 Then there exists a codimension one algebraic set $Y\subset
 \poly(T(\lambda_0))$ such that if
 $P \in \mathcal{C}(T(\lambda_0)) \setminus Y$ is post-critically
 finite, then there exists $f \in \mathcal{R}(\lambda_0)$ such that
 $\chi_{\lambda_0}(f)=P$.

 Furthermore, if $\lambda_0$ is post-critically finite, then such $f$ is
 unique.
\end{thm}
Note that the last part follows from the injectivity of
$\chi_{\lambda_0}$ (Theorem~\ref{thm-inj}).

The algebraic set $Y$ in the theorem is defined in a similar way as in
Theorem~\ref{lem-thickening}.
In particular, it is empty when $\lambda_0$ is primitive.
We call $f$ the \emph{tuning of $\lambda_0$ and $P$},
or when $f_0 \in \poly(T_0)$ satisfies $\lambda_{f_0}=\lambda_0$,
we also say $f$ is the \emph{tuning of $f_0$ and $P$}.
If $\lambda_0$ is post-critically finite, then such $f$ is also
post-critically finite.

\begin{lem}
 \label{lem-comb-tuning-primitive}
 Let $\lambda_0$ be a rational lamination over a mapping schema $T_0$
 and let $\lambda$ be a $T(\lambda_0)$-invariant rational lamination.
 
 If $\lambda$ is primitive and all periods of periodic
 $\lambda$-unlinked classes are sufficiently large, then the
 combinatorial tuning $\lambda_1$ of $\lambda_0$ and $\lambda$ is also
 primitive.
\end{lem}

\begin{proof}
 Assume $\lambda_1$ is not primitive,
 that is, there exist some $\lambda_1$-unlinked classes
 $L_1$ and $L_2$ and $\lambda_1$-class $A$
 such that $\overline{L_j} \cap A \ne \emptyset$ for $j=1,2$.

 If $L_1$ and $L_2$ lie in the same $\lambda_0$-unlinked class $M$,
 then there exists some $n\ge 0$ such that $v=m_{T_0}^n(M) \in
 |T(\lambda_0)|$ and $m_{T_0}^n:M \to v$ is a cyclic order
 preserving bijection.
 This implies that $L_1' = \alpha_v(m_{T_0}^n(L_1))$ and
 $L_2' = \alpha_v(m_{T_0}^n(L_2))$ are $\lambda$-unlinked class and there
 exists a $\lambda$-class $B$ intersecting both $\overline{L_1'}$ and
 $\overline{L_2'}$. Therefore $\lambda$ is not primitive, that is a
 contradiction.
 
 Let $M_j$ be the $\lambda_0$-unlinked class containing $L_j$.
 We have proved $M_1 \ne M_2$. We also have $\overline{M_j} \cap A \ne
 \emptyset$. 
 Let $B \subset A$ be the $\lambda_0$-class such that $\overline{M_1}
 \cap B \ne \emptyset$.
 Then there exists $\lambda_0$-unlinked class $M_3 \ne M_1$
 such that $\overline{M_3} \cap B$ are non-empty.
 Since the eventual periods of $A$ and $B$ are the same,
 there are only finite possibility for the eventual period of $A$ by
 Lemma~\ref{lem-non-primitive}.

 Therefore, if all periods of periodic $\lambda$-unlinked classes are
 sufficiently large, all periodic angles in the closures of
 $\lambda$-unlinked classes have periods greater than that of such $A$
 by Lemma~\ref{lem-same-period}.
 Hence the above argument shows that there are no such triple
 $(L_1,L_2,A)$, so $\lambda$ is primitive.
\end{proof}

\section{Continuity of straightening maps}
\label{sec-conti}
Now we give some sufficient condition for straightening maps to be
continuous at a given map $f \in \mathcal{R}(\lambda_0)$.

The argument on continuity of straightening maps by Douady and Hubbard
(Theorem~\ref{thm-quad-conti}, Lemma~\ref{lem-qc-conj}) can be applied
to our case. In particular, we have the following:
\begin{lem}
 \label{lem-conti-DH}
 Let $\lambda_0$ be a rational lamination over $T_0$.
 Assume $f_n$ converges to $f$ in $\mathcal{R}(\lambda_0)$ and 
 $\chi_{\lambda_0}(f_n)$ converges to $h \in \mathcal{C}(T(\lambda_0))$.
 Then there exist some $K \ge 1$ independent of $n$ and a $K$-quasiconformal
 hybrid conjugacy $\psi_n$ between a $\lambda_0$-renormalization of
 $f_n$ and $\chi_{\lambda_0}(f_n)$ 
 such that, by passing to a
 subsequence,  $\psi_n$ converges to a $K$-quasiconformal conjugacy
 $\psi$ between $\lambda_0$-renormalization of $f$ and $h$.
 In particular, $h$ and $\chi_{\lambda_0}(f)$ are quasiconformally
 equivalent.
\end{lem}

\begin{thm}
 \label{thm-conti}
 Let $\lambda_0$ be a rational lamination over $T_0$.
 If $f \in \mathcal{R}(\lambda_0)$ is quasiconformally
 rigid, then $\chi_{\lambda_0}$ is continuous at $f$.
\end{thm}
We say $f \in \poly(T_0)$ is \emph{quasiconformally rigid} if
any $g \in
\poly(T_0)$ quasiconformally conjugate to $f$ is affinely
conjugate to $f$.

\begin{proof}
 Assume $f_n$ converges to $f$ in $\mathcal{R}(\lambda_0)$.
 Then by the lemma above, we may assume that
 $\chi_{\lambda_0}(f_n)$ converges to $h$, which is quasiconformally
 conjugate to $\chi_{\lambda_0}(f)$.
 Hence it follows that there exists $\tilde{f} \in
 \mathcal{R}(\lambda_0)$ such that $\chi_{\lambda_0}(\tilde{f})=h$ and
 $\tilde{f}$ is quasiconformally conjugate to $f$
 \cite[Lemma~9.2]{Inou-Kiwi-straightening}.
 
 By assumption, we have $\tilde{f}=f$ and $h=\chi_{\lambda_0}(f)$.
\end{proof}

Similarly, we have some partial continuity of $\chi_{\lambda_0}^{-1}$.
Observe that when $\lambda_0$ is post-critically finite, 
$\chi_{\lambda_0}^{-1}:\chi_{\lambda_0}(\mathcal{R}(\lambda_0)) \to
\mathcal{R}(\lambda_0)$ is well-defined since $\chi_{\lambda_0}$ is injective.

\begin{prop}
 \label{prop-conti-inv}
 Under the assumption of Theorem~\ref{thm-conti},
 assume $\lambda_0$ is post-critically finite and 
 there exists a convergent sequence $f_n \to \tilde{f}$ in
 $\mathcal{R}(\lambda_0)$ such that $\chi_{\lambda_0}(f_n) \to
 \chi_{\lambda_0}(f)$.
 Then $\tilde{f}=f$.
\end{prop}

\begin{proof}
 By Lemma~\ref{lem-conti-DH}, $\chi_{\lambda_0}(f)$ and
 $\chi_{\lambda_0}(\tilde{f})$ are quasiconformally conjugate.
 Hence, by \cite[Lemma~9.2]{Inou-Kiwi-straightening}, there exists some
 $\hat{f}$ quasiconformally conjugate to $f$ such that
 $\chi_{\lambda_0}(\hat{f})=\chi_{\lambda_0}(\tilde{f})$.
 
 Therefore, $f = \hat{f}=\tilde{f}$ by assumption and Theorem~\ref{thm-inj}.
\end{proof}

The following proposition shows the continuity of
$\chi_{\lambda_0}^{-1}$ at Misiurewicz maps.
\begin{prop}
 \label{prop-Mis-proper}
 Let $\lambda_0$ be a post-critically finite rational lamination over
 $T_0$.
 Consider $\hat{f}, f_n \in \mathcal{R}(\lambda_0)$.
 Let $P=\chi_{\lambda_0}(\hat{f})$ and $P_n = \chi_{\lambda_0}(f_n)$.
 If $\hat{f}$ is Misiurewicz and $P_n \to P$, then $\lim_{n \to \infty}
 f_n = \hat{f}$.
\end{prop}

For the proof, we use the following proposition by Kiwi
\cite[Proposition~4.3]{Kiwi-combcont}:
\begin{prop}
 Consider a maximal $d$-invariant real lamination $\lambda$. Choose a
 finite set
 $F_1,\dots,F_k$ of $\lambda$-classes. Let $A={t_0,t_{q-1}}$ be a
 $\lambda$-class (subscripts respecting cyclic order and modulo $q$).
 Given $N>0$ and $\varepsilon>0$, there exist rational $\lambda$-classes
 $A_0,\dots,A_{q-1}$ such that for all $i$:
 \begin{enumerate}
  \item $A_i \subset (t_i,t_i+\varepsilon) \cup (t_{i+1}-\varepsilon,
	t_{i+1})$ and $A_i$ intersects both $(t_i,t_i+\varepsilon)$ and
	$(t_{i+1}-\varepsilon,t_{i+1})$.
  \item $A_i$ is disjoint from the grand orbit of $F_1 \cup \dots \cup
	F_k$ under $m_d$.
  \item $d^NA_i$ is not a periodic class.
 \end{enumerate}
\end{prop}
For a Misiurewicz map $f$, the equivalence classes of the \emph{real
extension} $\hat{\lambda}_f$ of $\lambda_f$ consists of
$\lambda_f$-classes and finite $\lambda_f$-unlinked classes, and
$\hat{\lambda}_f$ is \emph{maximal}, 
so we can apply this proposition for $\hat{\lambda}_f$.
(see \cite{Kiwi-combcont} for the
general definitions).

\begin{proof}[Proof of Proposition~\ref{prop-Mis-proper}]
 First recall that if an external ray $R_{\hat{f}}(\theta)$ land at a repelling
 periodic point, then the landing point of $R_f(\theta)$ moves
 continuously on $f$ close to $\hat{f}$ (see
 \cite[Lemma~B.1]{Goldberg-Milnor}). By taking inverse images, 
 this also holds for preperiodic eventually repelling point assuming it
 is not pre-critical.

 Therefore, for $\theta_1, \theta_2 \in \Q/\Z$ with $\theta_1
 \sim_{\lambda_f} \theta_2$, 
 if the common landing point of $R_{\hat{f}}(\theta_1)$ is not
 pre-critical,
 then 
 \[
  U_{(\theta_1,\theta_2)}=\{f \in \mathcal{C}(d);\ \theta_1
 \sim_{\lambda_f} \theta_2\}
 \]
 is a neighborhood of $\hat{f}$.
 Let $\lambda_{\hat{f}}^*$ be the set of
 $\lambda_{\hat{f}}$-equivalent pair $(\theta_1,\theta_2)$ such that the
 common landing point of $R_{\hat{f}}(\theta_1)$ and
 $R_{\hat{f}}(\theta_2)$ is not pre-critical.

 Then, now apply the previous proposition to $\hat{\lambda}_{\hat{f}}$,
 where the forbidden classes $F_1,\dots,F_k$ are the critical
 $\lambda_{\hat{f}}$-classes. Then it follows that 
 \[
  \bigcap_{(\theta_1,\theta_2) \in \lambda_{\hat{f}}^*}
 U_{(\theta_1,\theta_2)}
 = \{f \in \mathcal{C}(d);\ \lambda_f  \supset \lambda_{\hat{f}}\}
 = \{\hat{f}\}.
 \]
 In other words, $\{\}U_{(\theta_1,\theta_2)}\}$ is a neighborhood basis
 at $\hat{f}$.

 Now take any $(\theta_1,\theta_2) \in \lambda_{\hat{f}}^*$.
 If $\theta_1,\theta_2$ are $\lambda_0$-equivalent, then 
 $f_n \in U_{(\theta_1,\theta_2)}$ for any $n$.
 Otherwise, take the smallest $n\ge 0$ such that
 $(d^k\theta_1,d^n\theta_2)$ lies in $v \in |T(\lambda_0)|$ and let 
 $\eta_j = \alpha_v(d^k\theta_j)$.
 Then $\eta_1$ and $\eta_2$ are $\lambda_P$-equivalent, hence
 $\lambda_{P_n}$-equivalent for sufficiently large $n$.
 
 Since $\lambda_{f_n}$ is the combinatorial tuning of $\lambda_0$ and
 $\lambda_{P_n}$, it follows that $f_n \in U_{(\theta_1,\theta_2)}$.
 Therefore, we have $f_n \to \hat{f}$.
\end{proof}


\section{Parabolic bifurcation}
\label{sec-para-bif}

In this section and the next section,
we study perturbations in the connectedness locus
and see when a polynomial having a polynomial-like restriction satisfies
the perturbation condition in Theorem~\ref{thm-conti-mult} for
sufficiently many repelling periodic point.

In this section, we study parabolic bifurcations
and give a sufficient condition to have nice perturbations.
The successive section is devoted to the study of Misiurewicz bifurcations
to find parabolic polynomials satisfying this sufficient condition.

\begin{defn}
 We say a polynomial $f$ of degree $d \ge 3$ with connected Julia set
 satisfies \condI\ if the following hold;
 \begin{enumerate}
  \renewcommand{\theenumi}{(C1-\alph{enumi})}
  \item \label{condI-parab}
	0 is a non-degenerate $1$-parabolic periodic point of period
	$p$ for $f$;
  \item \label{condI-quad-like}
	there exists a quadratic-like restriction $f^p:V' \to V$ of
	$f^p$ containing $0$ hybrid equivalent to $z+z^2$;
  \item \label{condI-cp}
	let $\omega \in V'$ be the critical point of this quadratic-like
	restriction. 
	There exists another critical point $\omega'$ for $f$ and $N > 0$ 
	such that $f^n(\omega') \not \in K(f^p;V',V)$ for $n<N$ and
	$f^p(\omega)=f^N(\omega') \in K(f^p;V',V)$.
 \end{enumerate}
 We say $f$ satisfies \condII\ if it satisfies \condI\ and
 \begin{enumerate}
  \renewcommand{\theenumi}{(C2-\alph{enumi})}
  \item \label{condII-preperiodic}
	every critical point other than $\omega$ and $\omega'$ is
	preperiodic;
  \item \label{condII-primitive}
	the rational lamination $\lambda_f$ of $f$, which is
	post-critically finite by \ref{condII-preperiodic}, is primitive.
 \end{enumerate}
\end{defn}

\begin{rem}
 \label{rem-relax-preperiodicity}
 The condition \ref{condII-preperiodic} is just to obtain an analytic
 subset in the parameter space where a desired bifurcation occurs with
 keeping other dynamical properties.
 Therefore, for example, we can relax it to admit critical points in 
 bounded attracting basins.
 In this case, we can use the analytic dependence of the dynamics in the
 hyperbolic component \cite{Milnor-hyp} to get such an analytic subset.
\end{rem}

The following condition implies that we have nice perturbations to apply
Theorem~\ref{thm-conti-mult} (see the proof of Theorem~\ref{thm-discont}).
\begin{defn}
 Let $f$ satisfy \condI\ and let $\alpha$ be a repelling periodic point
 of $f$. 
 We say $f$ satisfy \condIII$_\alpha$ if
 there exists a convergent double sequence
 \[
  f_{n,m} \xrightarrow{m \to \infty} f_n \xrightarrow{n \to \infty} f
 \]
 in $\mathcal{C}(\lambda_f) (\subset \mathcal{C}(d))$ such that the
 following hold.
 Let us denote the continuations of critical points $\omega$ and
 $\omega'$ for $f_n$ and $f_{n,m}$ by $\omega_n$ and $\omega_{n,m}$
 respectively. 
 Similarly, let $\alpha_n$ and $\alpha_{n,m}$ be the continuations of
 the repelling periodic point $\alpha$ for $f_n$ and $f_{n,m}$
 (that is, we require that they do not bifurcate under these
 perturbations). 
 Let 
 \begin{align*}
  x_n&=f_n^p(\omega_n), & y_n&=f_n^N(\omega_n'), \\
  x_{n,m}&=f_{n,m}^p(\omega_{n,m}), & y_{n,m}&=f_{n,m}^N(\omega_{n,m}').
 \end{align*}
 (Recall that $\lim x_n = \lim y_n$ by \ref{condI-cp}.)
 \begin{enumerate}
  \renewcommand{\theenumi}{(C3-\alph{enumi})}
  \item \label{cond3-parab}
	0 is a periodic point of period $p$ for $f_n$ and $f_{n,m}$.
	It is non-degenerate and $1$-parabolic for $f_n$;
  \item \label{cond3-xy}
	$x_n \ne y_n$
	(hence $x_{n,m} \ne y_{n,m}$ for sufficiently large $m$).
  \item \label{cond3-crit-rel}
	the other critical orbit relations of $f$ are preserved for
	$f_{n,m}$ (hence also for $f_n$), i.e.,
	all critical points do not bifurcate under these perturbations
	and if $c, c' \in \crit(f) \setminus \{\omega,\omega'\}$
	(possibly $c=c'$) satisfy
	$f^k(c)=f^{k'}(c')$, then
	$f_{n,m}^k(c_{n,m})=f_{n,m}^{k'}(c_{n,m})$ for any $n,m$ where
	$c_{n,m}$ and $c_{n,m}'$ are the continuations of $c$ and $c'$
	respectively. 
  \item \label{cond3-quad-like}
	$f_{n,m}^p:V_{n,m}' \to V_{n,m}$ is a quadratic-like
	restrictions near $0$ and $\omega$, converging to
	a quadratic-like restriction $f_n^p:V_n' \to V_n$ locally
	uniformly, and it also converges to $f^p:V' \to V$ as $n \to
	\infty$.
	(Hence $f_n^p:V_n' \to V_n$ is hybrid equivalent to $z+z^2$.)
  \item \label{cond3-filled-julia}
	$x_n, y_n \in \Int K(f_n^p;V_n',V_n)$, and
	$x_{n,m}, y_{n,m} \in K(f_{n,m}^p;V_{n,m}',V_{n,m})$.
  \item \label{cond3-geom-conv}
	$f_{n,m}$ geometrically converges to $(f_n,g_n)$ as $m \to
	\infty$ such that $g_n(x_n)=\alpha_n$ and
	$g_n'(x_n) \ne 0$.
 \end{enumerate}
 We say $f$ satisfies \condIII\ if \condIII$_\alpha$ holds for any
 repelling periodic point $\alpha$ in $J(f^p;V',V)$.
\end{defn}
Note that we only consider $\alpha$ in the Julia set
$J(f^p;V',V)$ for the quadratic-like renormalization hybrid equivalent
to $z+z^2$ for \condIII.

\begin{rem}
 For simplicity, we say a polynomial satisfies \emph{\condI, \condII, or
 \condIII\ for a parabolic periodic point $x$}
 if a polynomial affinely conjugate to it by a conjugacy sending $x$ to
 $0$ does, because we mainly consider the space of monic centered
 polynomials.

 We can also define these conditions for a polynomial over a mapping
 schema in the same way. Observe that the mapping schema must have
 non-trivial Fatou critical relation in order to satisfy those
 conditions.
\end{rem}

Here, we prove the following.
\begin{thm}
 \label{thm-para-purturb}
 Let $\lambda_0$ be a post-critically finite
 $d$-invariant rational lamination with non-trivial Fatou critical relation.
 Assume $f \in \mathcal{R}(\lambda_0)$ satisfies \condII\
 and $\mathcal{R}(\lambda_f) \subset \mathcal{R}(\lambda_0)$.
 Then $f$ satisfies \condIII\ such that $f_n, f_{n,m} \in
 \mathcal{R}(\lambda_0)$ 
 and $\lambda_{f_n}=\lambda_{f}$ 
 for sufficiently large $n,m$.
 \end{thm}

The rest of this section is devoted to prove this theorem.
We first study the bifurcation of a quadratic polynomial
$Q(z)=z^2+1/4 \in \poly(2)$, which is affinely conjugate to $z+z^2$.

Consider a repelling periodic point $\alpha(Q)$ of $Q$ and let $\theta$
be the landing angle for $\alpha(Q)$.
Let $c_m$ be the landing point of the parameter ray
$\mathcal{R_M}(\theta/2^m)$ for the Mandelbrot set and let
$Q_m(z)=z^2+c_m$. 
Let $\alpha(Q_m)$ be the landing point of the
external ray $R_{Q_m}(\theta)$, which is the repelling periodic point
and $\alpha(Q_m) \to \alpha(Q)$ as $m \to \infty$.
The critical point $0$ is preperiodic under $Q_m$
because $c_m=Q_m(0)$ is the landing point of $R_{Q_m}(\theta/2^m)$
\cite{Douady-Hubbard-etude}, it follows that $Q_m^{m+1}(0) =
\alpha(Q_m)$. Hence $Q_m$ is Misiurewicz.

\begin{lem}
 \label{lem-Q-conv}
 There exists some Lavaurs map $g_Q$ such that $Q_m \xrightarrow{\geom}
 (Q,g_Q)$ with $g_Q(Q(0))=\alpha(Q)$ and $g_Q'(Q(0)) \ne 0$.
\end{lem}

\begin{proof}
 Since the Mandelbrot set is locally connected at $1/4$
 \cite{Hubbard-Yoccoz}, $c_m \to 1/4$ as $m \to \infty$.
 Furthermore, this convergence is tangential to the positive real axis,
 hence it follows that a Fatou coordinate $\Phi_{Q_m}$ is defined for
 sufficiently large $m>0$.

 By the continuity of Fatou coordinates, 
 there exists some $k>0$ such that the landing point of
 $R_{Q_m}(\theta/2^k)$ is contained in the domain
 of definition of $\Phi_{Q_m}$ for sufficiently large $m>0$.
 We may also assume $c_m$ is also contained in the domain of definition
 of $\Phi_{Q_m}$
 because we can extend $\Phi_{Q_m}$ by the functional equation
 $\Phi_{Q_m}(Q_m(z))=\Phi_{Q_m}(z)+1$. 
 Since $\Phi_{Q_m}$ has a critical point only at the backward orbit of
 the critical point, $\Phi_{Q_m}$ is univalent on an open set containing
 $1/4$ and $c_m(\approx 1/4)$ of a definite size.
 Hence $Q_m^m(c_m)=Q_m^k\circ \Phi_{Q_m}^{-1}( \Phi_{Q_m}(z)+m-k)$
 is well-defined and univalent near $z=c_m$ because $Q_m^k$ is univalent
 on a neighborhood of $\Phi_{Q_m}^{-1}(\Phi_{Q_m}(z)+m-k)$, which is the
 landing point of $R_{Q_m}(\theta/2^k)$, of a definite size.
 
 Therefore, $Q_m^n(z) = Q_m^k \circ \Phi_{Q_m}^{-1}(\Phi_{Q_m}(z)+m-k)$
 converges to a Lavaurs map $g_Q$ as $m \to \infty$, which is univalent
 near $1/4$. 
\end{proof}

Take a sequence $\{\theta_n\} \subset \Q/\Z$ such that $\theta_n
\to \theta$,
and let $\alpha_n(Q_m)$ be the landing point of the
external ray $R_{Q_m}(\theta_n)$.
Since $Q_m$ is Misiurewicz, $J(Q_m)$ is locally connected, hence
$\alpha_n(Q_m) \to \alpha(Q_m)$ as $n \to \infty$.
Let $y_n(Q_m)$ be the landing point of $R_{Q_m}(\theta_n/2^m)$.
\begin{figure}[th]
 \fbox{\includegraphics[width=6cm]{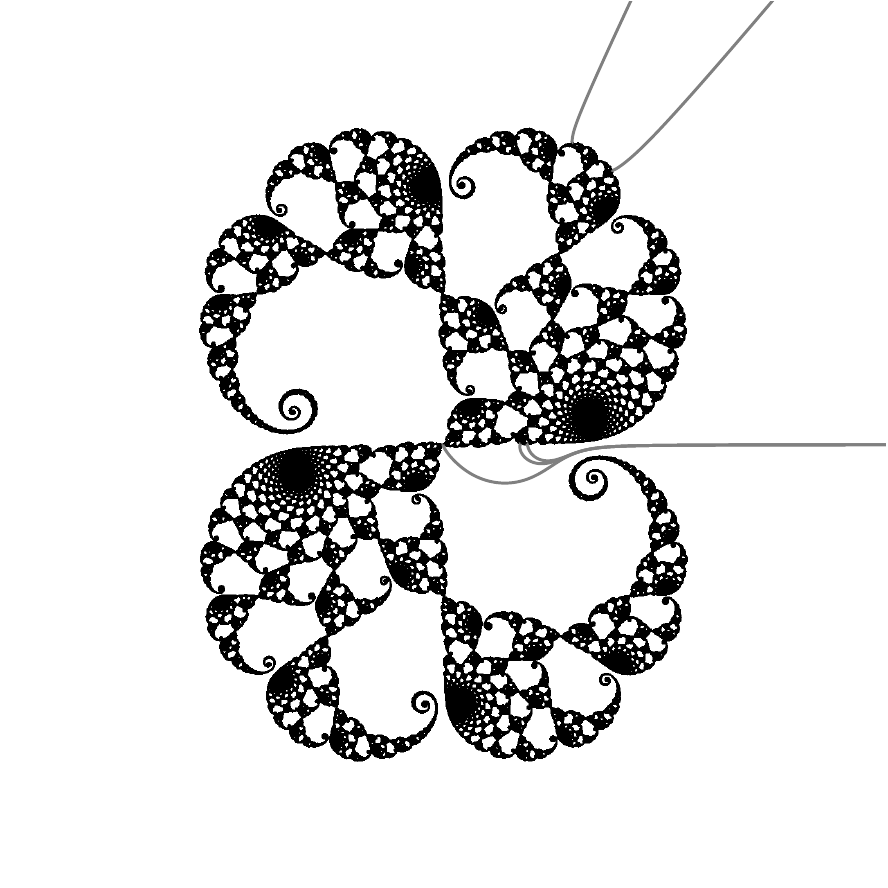}}
 \fbox{\includegraphics[width=6cm]{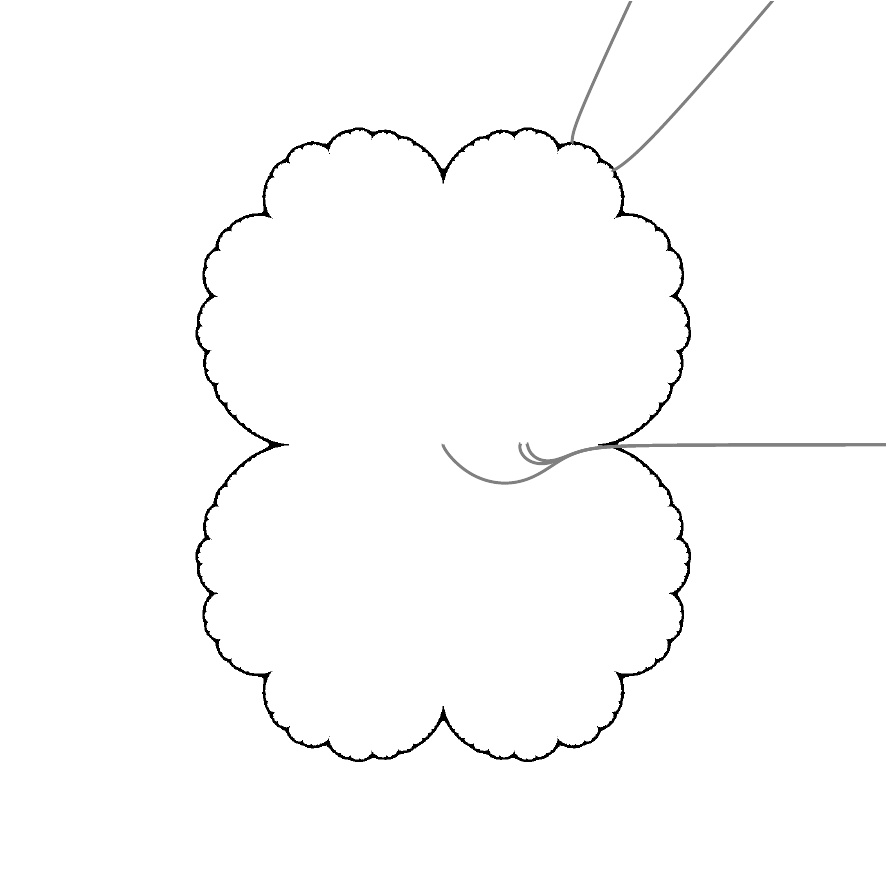}}
 \caption{The external rays of angle $\theta$, $\theta_n$, $\theta/2^m$,
 $\theta_n/2^m$ and $\theta/2^{m+1}$ for $Q_m$ and their limits.
 By construction, $R_{Q_m}(\theta)$ and $R_{Q_m}(\theta_n)$ land at
 repelling (pre)periodic points $\alpha(Q_m)$ and $\alpha_n(Q_m)$
 respectively, and $R_{Q_m}(\theta/2^m)$ and $R_{Q_m}(\theta/2^{m+1})$
 land at the critical value and critical point respectively.}
 \label{fig-Q}
\end{figure}

\begin{lem}
 The limit 
 \[
  y_n(Q)=\lim_{m \to \infty} y_n(Q_m)
 \]
 exists and
 \[
  \lim_{n \to \infty} y_n(Q) = 1/4. 
 \]
\end{lem}

See Figure~\ref{fig-Q}. 
Since the critical value $Q_m(0)$ is the landing point of
$R_{Q_m}(\theta/2^m)$, we have 
\[
 \lim_{m \to \infty} \lim_{n \to \infty} y_n(Q_m)
 = \lim_{m \to \infty} Q_m(0) = Q(0)=1/4. 
\]
However, one cannot change the order of limits in general,
because the Julia set does not move continuously at parabolic maps.

\begin{proof}
 As we saw in the previous lemma, we have $Q_m^m \longrightarrow g_Q$
 near $1/4$ and $g_Q(1/4) = \alpha(Q)$.
 Since $g_Q'(1/4) \ne 0$, there exists an inverse branch $h$ defined
 near $\alpha(Q)$ with $h(\alpha(Q))=1/4$.
 
 Therefore, for sufficiently large $m$, there exists an inverse branch $h_m$
 of $Q_m^m$ defined near $\alpha(Q)$ which satisfies $h_m(\alpha(Q_m)) =
 1/4$. 

 If $n$ is sufficiently large, $\alpha_n(Q_m)$ is close to
 $\alpha(Q_m)$, hence by construction, $y_n(Q_m) = h_m(\alpha_n(Q_m))$.
 Therefore, $y_n(Q) = h(\alpha_n(Q))$ where $\alpha_n(Q)$ is the landing
 point of $R_Q(\theta_n)$. 

 Since $J(Q)$ is locally connected, $\alpha_n(Q) \to \alpha(Q)$ as $n
 \to \infty$, thus $y_n(Q) = h(\alpha_n(Q)) \to h(\alpha(Q)) = 1/4$.
\end{proof}

Let $f$ satisfy the assumption of the Theorem~\ref{thm-para-purturb}.
By assumption, the reduced mapping schema 
$T(\lambda_f)=(|T(\lambda_f)|,\sigma,\delta)$ is equal to
$T_{\capt,d_1+1}$, the schema of capture type of total degree $d_1+1$.

Define a polynomial $\tilde{Q}_{n,m}$ over $T(\lambda_f)$ as follows:
\[
 \tilde{Q}_{n,m}(v_i,z)=
 \begin{cases}
  (v_0, Q_m(z)) & \mbox{if }i=0, \\
  (v_0, z^{d_1}+y_n(Q_m)) & \mbox{if }i=1.
 \end{cases}
\]
Then $\tilde{Q}_n=\lim_{m \to \infty} \tilde{Q}_{n,m}$
and $\tilde{Q}=\lim_{n \to \infty} \tilde{Q}_n$ exist and satisfy the
following:
\begin{align*}
 \tilde{Q}_n(v_i,z) &= 
 \begin{cases}
  (v_0,Q(z))& \mbox{if }i=0, \\
  (v_0, z^{d_1}+y_n(Q)) & \mbox{if }i=1,
 \end{cases} \\
 \tilde{Q}(v_i,z) &= 
 \begin{cases}
  (v_0,Q(z))& \mbox{if }i=0, \\
  (v_0, z^{d_1} + \frac{1}{4}) & \mbox{if }i=1.
 \end{cases}
\end{align*}

Since $\tilde{Q}_{n,m}$ is Misiurewicz and $\lambda_f$ is primitive,
there exists $f_{n,m} \in \mathcal{R}(\lambda_f)=\mathcal{C}(\lambda_f)$ such
that $\chi_{\lambda_f}(f_{n,m})=\tilde{Q}_{n,m}$ by
Theorem~\ref{thm-pcf-tuning}.

By taking a subsequence, we may assume
\[
 f_{n,m} \xrightarrow{m \to \infty} f_n \xrightarrow{n \to \infty}
 \hat{f}
\]
for some $f_n$ and $\hat{f}$.
Since $\mathcal{C}(\lambda_f)=\mathcal{R}(\lambda_f)$ is compact by
Theorem~\ref{thm-cpt},
$f_n$ and $\hat{f}$ are also $\lambda_f$-renormalizable.

Let $(f_n^{\ell_i}:U_{n,v_i}' \to U_{n,v_0})_{i=0,1}$ and
$(\hat{f}^{\ell_i}:\hat{U}_{v_i}' \to \hat{U}_{v_0})_{i=0,1}$ be
$\lambda_f$-renormalizations of $f_n$ and $\hat{f}$.
Since the straightening map is continuous for quadratic-like families,
the quadratic-like restrictions $f_n^{\ell_0}:U_{n,v_0}' \to U_{n,v_0}$
and $\hat{f}^{\ell_0}:\hat{U}_{v_0}' \to \hat{U}_{v_0}$ are hybrid
equivalent to $Q(z)=z^2+1/4$.
Let $\omega_n$ and $\omega_n'$ (resp.\ $\hat{\omega}$ and
$\hat{\omega}'$) be the critical points for $f_n$ (resp.\ $\hat{f}$)
lying in $U_{n,v_0}'$ and $U_{n,v_1}'$ (resp.\ $\hat{U}_{v_0}'$ and
$\hat{U}_{v_1}'$) respectively.
Note that $\ell_0=p$ and $\ell_1=N$.

\begin{lem}
 \label{lem-crit}
 \begin{enumerate}
  \item \label{item-crit-interior}$ f_n^N(\omega_n') \in \Int
	K_{f_n}(v_0)$ for sufficiently large $n$.
  \item \label{item-crit-rel}
	$\hat{f}^N(\hat{\omega}') = \hat{f}^p(\hat{\omega})$.
 \end{enumerate}
\end{lem}

\begin{proof}
 Since $\mathcal{R}(\lambda_f)$ is compact, we may assume that 
 the hybrid conjugacy $\psi_{n,m}$ between
 $\lambda_f$-renormalization of $f_{n,m}$ and
 $\tilde{Q}_{n,m}=\chi_{\lambda_f}(f_{n,m})$ are uniformly
 $K$-quasiconformal by Lemma~\ref{lem-conti-DH}.
 By passing to a subsequence, we may further assume that 
 \[
  \psi_{n,m} \xrightarrow{m \to \infty} \varphi_n 
  \xrightarrow{n \to \infty} \hat{\varphi}
 \]
 and they are all $K$-quasiconformal.
 Then $\varphi_n$ conjugates the $\lambda_n$-renormalization of $f_n$
 to $\tilde{Q}_n$ and $\hat{\varphi}$
 conjugates that of $\hat{f}$ to $\tilde{Q}$.
 Hence it follows that 
 \begin{align*}
  \hat{\varphi}(\hat{f}^N(\hat{\omega}'))
  &= \tilde{Q}(\hat{\varphi}(\hat{\omega}')) 
  = (v_0,1/4) = \hat{\varphi}(\hat{f}^p(\hat{\omega})),
 \end{align*}
 so we have \ref{item-crit-rel} and $\hat{f}^N(\hat{\omega}') \in \Int
 K_{\hat{f}}(v_0)$. Therefore, \ref{item-crit-interior} also follows by
 continuity.
 Note that $K(\tilde{Q}_n,v_0)=K(Q)$ does not depend on $n$.
\end{proof}

\begin{lem}
 \label{lem-lambda-f_n}
 $\lambda_{f_n}=\lambda_{\hat{f}}=\lambda_{f}$ for sufficiently large $n$.
\end{lem}

\begin{proof}
 We use the same notation as in the proof of the previous lemma.
 By the previous lemma, the critical points of $\tilde{Q}_n$ and
 $\tilde{Q}$ lie in the interior of the filled Julia set.
 Since $K(\tilde{Q}_n,v_0) = K(\tilde{Q},v_0) = K(Q)$, 
 it follows that the rational laminations of $\tilde{Q}_n$ and $\tilde{Q}$
 are trivial. By Theorem~\ref{thm-comb-tuning}, the rational
 laminations of $f_n$ and $\hat{f}$ are the combinatorial tuning of
 $\lambda_f$ and the trivial rational lamination, which is equal to
 $\lambda_f$ itself.
\end{proof}

\begin{lem}
 $\hat{f}=f$.
\end{lem}

\begin{proof}
 The quasiconformal rigidity of $\tilde{Q}$ implies that
 $\chi_{\lambda_f}(\hat{f})=\chi_{\lambda_f}(f)= \tilde{Q}$.
 Therefore, the lemma follows from the injectivity of the straightening
 map $\chi_{\lambda_f}$.
\end{proof}

\begin{proof}[Proof of Theorem~\ref{thm-para-purturb}]
 We have already constructed a convergent double sequence
 \[
 f_{n,m} \xrightarrow{m \to \infty} f_n \xrightarrow{n \to \infty}
 f
 \]
 in $\mathcal{R}(\lambda_f) \subset \mathcal{R}(\lambda_0)$,
 hence it is enough to check that this satisfies the conditions in \condIII.

 The condition \ref{cond3-parab} holds by changing the coordinate if necessary.
 The condition \ref{cond3-xy} follows from the fact that
 $\varphi_n(x_n)=\tilde{Q}_n(v_0,0) \ne 
 \varphi_n(y_n)=\tilde{Q}_n(v_1,0)$ by construction.
 Since all critical points except $\omega$ and $\omega'$ lie in the
 Julia set and preperiodic, their behavior is described in terms of
 $\lambda_f$. Thus \ref{cond3-crit-rel} follows because $f_{n,m}$ and
 $f_n$ admits $\lambda_f$.
 The $\lambda_f$-renormalizability of $f_{n,m}$ and $f_n$ and
 Lemma~\ref{lem-crit} imply
 \ref{cond3-quad-like} and \ref{cond3-filled-julia}.

 Let $g_Q$ be the Lavaurs map in Lemma~\ref{lem-Q-conv}.
 Then $g_Q(y_n(Q))=\alpha_n(Q)$ by construction.
 Since $Q_m^m(w)$ ($0 \le k \le m$)
 is sufficiently close to $K(Q)$ for $w$ sufficiently
 close to $1/4$, 
 we have 
 \[
 f_{n,m}^{mp} (z) 
 = \psi_{n,m}^{-1} \circ \tilde{Q}_{n,m}^m(\psi_{n,m}(z))
 \to \varphi_n^{-1} \circ \tilde{g}_n^Q(\varphi_n(z))
 \]
 for $z$ sufficiently close to $x_n$, 
 where $\tilde{g}_n^Q(v_0,w)=\lim_{m \to \infty}
 \tilde{Q}_{n,m}^m(v_0,w)=(v_0,g_Q(w))$.
 Therefore, $f_{n,m}^{mp}$ converges to a Lavaurs map 
 $g_n := \varphi_n^{-1}\circ \tilde{g}_Q \circ \varphi_n$.
 Since $\varphi_n$ is quasiconformal and 
 $g_Q'(Q(0)) \ne 0$, we have $g_n'(x_n) \ne 0$.
 Moreover,
 \[
  g_n(x_n)=\varphi_n^{-1} (v_0,g_Q(Q(0))) =
 \varphi_n^{-1}(v_0,\alpha(Q)) = \alpha_n.
 \]
 Therefore, we have proved \ref{cond3-geom-conv}.
\end{proof}


\section{Misiurewicz bifurcation}
\label{sec-mis-bif}

In this section, we prove the following:
\begin{thm}
 \label{thm-mis-perturb}
 Let $\lambda_0$ be a post-critically finite $d$-invariant rational
 lamination with a non-trivial Fatou critical relation and let $f_0 \in
 \mathcal{R}(\lambda_0)$ be a Misiurewicz $\lambda_0$-renormalizable
 polynomial.
 
 Then there exists a polynomial $f \in \mathcal{R}(\lambda_0)$
 arbitrarily close to $f_0$ such that 
 \begin{enumerate}
  \item $f$ satisfies \condII.
  \item $\mathcal{R}(\lambda_f) \subset \mathcal{R}(\lambda_0)$.
 \end{enumerate}
\end{thm}

The main difficulties of the proof of this theorem are the following.
First, all perturbation must be done inside $\mathcal{R}(\lambda_0)$.
In order to do this, we perturb in $\mathcal{C}(T(\lambda_0))$ and use
tuning to get nice perturbations.

Secondly, tuning is not defined everywhere, nor continuous.
We need to study tuning and straightening of
perturbations of parabolic maps, where discontinuity might occur.
Hence we approximate parabolic maps by Misiurewicz maps with a help of
combinatorial continuity by Kiwi \cite{Kiwi-combcont} and apply
results in Section~\ref{sec-conti} to show that the limiting parabolic
map is close to the original map.

\begin{rem}
 Similar to Remark~\ref{rem-relax-preperiodicity}, 
 what we essentially need is a two-dimensional analytic set passing
 through $f_0$, on which we have two active critical points, and the
 rest of the dynamics behaves stable.
\end{rem}

\subsection{Critical portraits and combinatorial continuity}
\label{subsec-combcont}

Here we briefly recall the notion of critical portraits and 
the combinatorial continuity for Misiurewicz maps.
\begin{defn}
 Let $T$ be a mapping schema.
 A \emph{critical portrait} over $T$ is a collection of sets
 $\Theta = \{\Theta_1,\dots,\Theta_m\}$ such that
 each $\Theta_j$ is contained in a fiber $\{v_j\} \times \R/\Z$ for some
 $v_j \in |T|$ and 
 \begin{enumerate}
  \renewcommand{\theenumi}{\rm (CP\arabic{enumi})}
  \item for every $j$, $\#\Theta_j \ge 2$ and
	$\#m_{T}(\Theta_j)=1$;
  \item $\Theta_1,\dots,\Theta_m$ are pairwise unlinked;
  \item \label{item-portrait-deg}
	For each $v \in |T|$, 
	$\displaystyle \sum_{j:v_j=v}(\#\Theta_j - 1) = \delta(v)-1$.
 \end{enumerate}
 We say a critical portrait $\Theta$ is \emph{preperiodic} if 
 all elements in $\Theta_1,\dots,\Theta_m$ are preperiodic by
 $m_T(v,\theta)= (\sigma(v),m_{\delta(v)}(\theta))$.
 
 For a Misiurewicz polynomial $P$ over $T$, 
 $\Theta$ is a \emph{critical portrait of $P$} 
 if for each $j$, there exists a critical point $\omega_j$ such that
 $R_f(v,\theta)$ lands at $\omega_j$ for any $(v,\theta) \in \Theta_j$. 
 In this case, $\Theta$ is always preperiodic.

 We endow the space of all critical portraits over $T$ with the
 \emph{compact-unlinked topology}, which is generated by the
 subbasis formed by
 \[
  V_X = \{\Theta=\{\Theta_j\};\ X\mbox{ is unlinked with }\Theta_j
 ~(\forall j)\}
 \]
 where $X$ is a closed subset of $\{v\} \times \R/\Z$ for some $v \in
 |T|$.

 Let us denote by $\mathcal{A}(T)$ the set of critical portraits over
 $T$ and let 
 \[
  \preper(\mathcal{A}(T))=\{\Theta \in \mathcal{A}(T);
  \mbox{ preperiodic}\}.
 \]
\end{defn}

For each critical portrait $\Theta$, we can naturally associate 
the \emph{impression} of $\Theta$, which is a set
$I(\Theta) \subset \mathcal{C}(T) \cap \overline{\mathcal{S}(T)}$,
where $\mathcal{S}(T)$ is the \emph{shift locus}, i.e., the set of
polynomials over $T$ with all critical points escaping
\cite{Kiwi-combcont}. 
\begin{thm}[Kiwi]
 \label{thm-combcont}
 If a critical portrait $\Theta \in \mathcal{A}(T)$ is preperiodic,
 then $I(\Theta)$ consists of a unique Misiurewicz polynomial $f_\Theta$ over
 $T$ such that $\Theta$ is a critical portrait of $f_\Theta$.

 Moreover, a map
 \[
 \preper(\mathcal{A}(T)) \longrightarrow \Mis(\poly (T)), \quad 
 \Theta \longmapsto f_\Theta \in \poly(T),
 \]
 is well-defined and continuous,
 where $\Mis(\poly(T))=\{f \in \poly(T);\mbox{ Misiurewicz}\}$.
\end{thm}

\begin{proof}
 The first part is proved by Kiwi \cite[Theorem~1,
 Corollary~5.3]{Kiwi-combcont}.
 Hence we need only show the continuity of $\Theta \mapsto f_\Theta$.
 By definition, $f \in I(\Theta)$ if there is a sequence of maps
 $\{f_n\}$ in the \emph{visible} shift locus, which is a dense subset of
 the shift locus, such that $f_n \to f$ and $\Theta(f_n) \to \Theta$.
 
 Consider a sequence $\Theta_n \to \Theta$ in $\preper(\mathcal{A}(T))$.
 Take $f_{n,m}$ in the visible shift locus with $f_{n,m} \to
 f_{\Theta_n}$ and $\Theta(f_{n,m}) \to \Theta_n$, where
 $\Theta(f_{n,m})$ is the critical portrait of $f_{n,m}$.
  
 It is easy to see that $\mathcal{A}(T)$ is first-countable.
 Take a countable neighborhood basis $\{U_k\}_{k \in \N}$ at $\Theta$ in
 $\mathcal{A}(T)$.
 Let $N_k>0$ be such that $\Theta_n \in U_k$ for $n \ge N_k$.

 For any $\varepsilon>0$, choose $m_n$ for each $n$ such that
 \begin{itemize}
  \item $f_{n,m_n}$ is $\frac{\varepsilon}{2}$-close to $f_{\Theta_n}$;
  \item $f_{n,m_n}$ is $\frac{1}{2^n}$-close to the connectedness locus;
  \item $\Theta(f_{n,m_n}) \in U_k$ for $n \ge N_k$.
 \end{itemize}
 Then $\Theta(f_{n,m_n}) \to \Theta$, so $f_{n,m_n} \to I(\Theta) =
 \{f_\Theta\}$.
 Hence it follows that the distance of $f_\Theta$ and $f_{\Theta_n}$ is
 less than $\varepsilon$ for sufficiently large $n$.
 Therefore, $\Theta \mapsto f_\Theta$ is continuous.
\end{proof}

\subsection{Perturbation in the target space}
\label{subsec-perturb}

To prove Theorem~\ref{thm-mis-perturb}, we first construct some 
nice perturbations in the target space, i.e., we perturb
$P_0=\chi_{\lambda_0}(f_0)$.

We apply the theorem on universality of the Mandelbrot set by
McMullen \cite{McMullen-muniv}.
For a mapping schema $T$, let $T^{\per}=(|T^{\per}|,\sigma,\delta)$ be a
sub-schema such that $|T^{\per}|=\{v \in |T|;\ $periodic$\}$.
Then we have natural projection $\pi:\poly(T) \to \poly(T^{\per})$.
We call $\pi(P)$ the \emph{periodic part of $P$} for $P \in \poly(T)$.
\begin{thm}
 \label{thm-muniv}
 Let $T$ be a mapping schema and 
 consider a holomorphic one-parameter family $(P_\mu)_{\mu \in \Delta}$
 in $\poly(T)$ parameterized by the unit disk $\Delta$.
 Let $\mathcal{B}$ be the bifurcation locus of the family
 $(\pi(P_{\mu}))$ of the periodic parts of $(P_{\mu})$.
 Then either $\mathcal{B}$ is empty or 
 there exists a quasiconformal image $\mathcal{M}'$
 of $\mathcal{M}_\delta$ for some $\delta \ge 2$ whose boundary is contained in
 $\mathcal{B}$
 where $\mathcal{M}_\delta$ is the connectedness locus of the unicritical
 family $\{z^\delta+c;$ $c \in \C\}$ of degree $\delta$.
 
 More precisely, there exists some $n>0$ such that for $\mu$ in
 $\mathcal{M}'$, there exist a critical point $\omega_\mu$
 analytically parameterized by $\mu$ and a polynomial-like restriction
 $P_\mu^n:W_\mu' \to W_\mu$ hybrid equivalent to $z^\delta+c(\mu)$
 such that 
 \begin{itemize}
  \item $\omega_\mu \in W_\mu'$,
  \item the local degree of $P_\mu$ at $\omega_\mu$ is equal to
	$\delta$, and 
  \item $c:\mathcal{M'} \to \mathcal{M}_d$ extends to a quasiconformal
	map of the plane.
 \end{itemize}
\end{thm}

\begin{proof}
 This is a simple generalization of \cite[Theorem~1.1,
 Theorem~4.1]{McMullen-muniv}.
 For $\mu \in \Delta$, let us denote $P_\mu^n(v,z) =
 (\sigma^n(v),P_{\mu,v}^n(z))$.
 If $\mathcal{B}$ is nonempty, then there exists some periodic $v \in
 |T|$ by $\sigma$ such that the family 
 $(P_{\mu,v}^p)_{\mu \in \Delta}$ has nonempty bifurcation locus where
 $p$ is the period of $v$.
 Therefore, it contains the quasiconformal image of $\partial
 \mathcal{M}_\delta$, for some $\delta \ge 2$.
 Since $\mathcal{B}$ contains the bifurcation locus of $(P_{\mu,v}^p)$,
 the theorem follows by \cite[Theorem~4.1]{McMullen-muniv}.
\end{proof}

Now consider the family of polynomials over a mapping schema $T$ of
total degree $d=\delta(T)$ with
all critical points marked (counted with multiplicity);
\[
 \widehat{\poly}(T) =
 \{(P,(v_1,\omega_1),(v_2,\omega_2),\dots,(v_{d-1},\omega_{d-1}));\ 
 P \in \poly(T),\ \crit(P)=\{(v_j,\omega_j)\}\}.
\]

Consider a Misiurewicz polynomial $(P_0,(v_j,\omega_{0,j})) \in
\widehat{\poly}(T(\lambda_0))$ and its neighborhood $\mathcal{U}$.
Let $\Theta_0$ be a critical portrait of $P_0$.
Take a preperiodic critical portrait $\Theta$ close to $\Theta_0$ such that
\begin{itemize}
 \item $\Theta=\{\Theta_1,\dots,\Theta_{d-1}\}$, i.e.,
       $\#\Theta_j=2$ for each $j$;
 \item there exists some $N'>1$ such that
       $m_T(\Theta_1)=m_T^{N'}(\Theta_2)$ and $m_T^{N'-1}(\Theta_2)
       \not\subset \Theta_1$;
 \item let $p_j$ be the eventual period of $\Theta_j$ by $m_T$.
       then $p_1(=p_2),  p_3, \dots, p_{d-1}$ are mutually different.
\end{itemize}
Let $(P_\Theta,(v_j,\omega_{\Theta,j})) \in
\widehat{\poly}(T(\lambda_0))$ be the polynomial having $\Theta$ as a
critical portrait and the landing point of the external ray of angle
$(v_j,\theta) \in \Theta_j$ is $(v_j,\omega_j)$.
Observe that there are no multiple critical points, i.e.,
$(v_j,\omega_j)$ are mutually different because otherwise both the eventual
periods and the preperiods must coincide.
If $\Theta$ is sufficiently close to $\Theta_0$, then $P_\Theta \in
\mathcal{U}$ by Theorem~\ref{thm-combcont}.
Moreover, we have
\begin{align*}
 P_\Theta^{N'}(v_2,\omega_{\Theta,2}) &=P_\Theta(v_1,\omega_{\Theta,1}), &
 P_\Theta^{N'-1}(v_2,\omega_{\Theta,2}) &\ne (v_1,\omega_{\Theta,1}).
\end{align*}
In fact, the first equality is trivial and
if $P_\Theta^{N'-1}(v_2,\omega_{\Theta,2}) = (v_1,\omega_{\Theta,1})$, 
then all of the three external rays of angles in $\Theta_1 \cup
m_T^{N'-1}(\Theta_2)$ land at $(v_1,\omega_{\Theta,1})$ and are mapped
to the same ray $R_{f_\Theta}(\theta)$, where $\{\theta\}= m_T(\Theta_1)
= m_T^{N'}(\Theta_2)$. 
This implies that $(v_1,\omega_{\Theta,1})$ is a multiple critical
point, so it is a contradiction.

Since each critical point is preperiodic, 
$P_{0,v_j}^{n_j}(\omega_j)$ is a repelling periodic point, say $p_j$,
for $j=1,\dots,d-1$. 
Let $p_j(P)$ be the continuation of $p_j$ as a repelling periodic point
for $P$ close to $P_0$, and let
\begin{equation}
 \label{eqn:defining Misiurewicz}
 h_j(P,\omega_j)=P_{v_j}^{n_j}(\omega_j)-p_j(P).
\end{equation}
Consider a local analytic set near $(P_0,(v_j,\omega_{0,j}))$:
\begin{equation}
 \label{eqn-X}
  \mathcal{X}=\{(P,(v_j,\omega_j)) \in \widehat{\poly}(T(\lambda_0));~
  h_3(P,\omega_3)=\cdots =h_{d-1}(P,\omega_{d-1})=h(P,\omega_1,\omega_2)=0
  \},
\end{equation}
where $h(P,\omega_1,\omega_2)=P^{N'}(v_2,\omega_2)-P(v_1,\omega_1)$.

\begin{lem}
 $\dim \mathcal{X}=1$.
\end{lem}

\begin{proof}
 Consider the following local analytic sets:
\begin{align*}
 \mathcal{X}' &=\{h_3(P,\omega_3)=\dots=h_{d-1}(P,\omega_{d-1})=0\}, \\
 \mathcal{X}'' &= \{h_1(P,\omega_1) = h_3(P,\omega_3) = \dots =
 h_{d-1}(P,\omega_{d-1}) = h(P,\omega_1,\omega_2) = 0\}
 \end{align*}
 Since $\mathcal{X}'' = \{h_1=h_2=\dots=h_{d-1}=0\}$, we have
 $\dim \mathcal{X}'_0=2$ and $\dim
 \mathcal{X}''_0=0$ by \cite{vanStrien-transversality}. 
 Since $\mathcal{X}''_0 \subset \mathcal{X}_0 \subset \mathcal{X}'_0$
 and $\dim \mathcal{X}'_0- \dim \mathcal{X}_0, \dim \mathcal{X}_0 - \dim
 \mathcal{X}''_0 \le 1$, the dimension of $\mathcal{X}_0$ is one.
\end{proof}

Observe that since all critical points are simple for $P_\Theta$, 
the natural projection from $\widehat{\poly}(T_0)$ to
$\poly(T_0)$ is a local isomorphism at $P_\Theta$.
Hence we identify them to simplify the notation.

Let $\mathcal{B}$ be the bifurcation locus of the periodic parts in
$\mathcal{X}$.
Then since $P_\Theta$ is Misiurewicz and a free critical point $\omega_1$
is contained in the periodic part, $P_\Theta \in \mathcal{B}$.
In particular, $\mathcal{B}$ is nonempty.
Therefore, by Theorem~\ref{thm-muniv},
There exists a copy $\mathcal{M}' \subset \mathcal{U} \cap
\mathcal{X}$ of the Mandelbrot set $\mathcal{M}=\mathcal{M}_2$, for
all critical points are simple.
Let $\xi:\mathcal{M}' \to \mathcal{M}$ be the homeomorphism defined by
straightening.
Let $P_1$ be the center of $\mathcal{M}'$, i.e.,
the quadratic-like restriction $P_1^p:W_1' \to W_1$ is hybrid
equivalent to $z^2$ (i.e., $\xi(P_1)=z^2$).

\begin{lem}
 We can take $\mathcal{M}'$ so that $P_1$ (equivalently,
 $\lambda_{P_1}$) is primitive and $p$ is arbitrarily large.
\end{lem}

\begin{proof}
 Otherwise, take a small copy $\mathcal{M}''
 \subset \mathcal{M}'$ which corresponds to a primitive copy of
 sufficiently high period in $\mathcal{M}$.
 Then the rational lamination of the center $P_2 \in
 \mathcal{M}''$ is the combinatorial tuning of $\lambda_{P_1}$ and a
 primitive rational lamination over $T(\lambda_{P_1})=T_{\capt}$.
 Hence $P_2$ is primitive by Lemma~\ref{lem-comb-tuning-primitive}.

 Therefore, the lemma is obtained by replacing $\mathcal{M}'$ by
 $\mathcal{M}''$.
\end{proof}

Therefore, we have proved the following. 
\begin{lem}
 \label{lem-P}
 Let $P_0$ be a Misiurewicz polynomial over a mapping schema $T$.
 For any neighborhood $\mathcal{U}$ of $P_0$, there exist a
 one-dimensional algebraic subset $\mathcal{X} \subset
 \widehat{\poly}(T)$ and a small copy of the Mandelbrot set
 $\mathcal{M}' \subset \mathcal{X} \cap \mathcal{U}$ such that
 \begin{enumerate}
  \item $\mathcal{X}$ is a local analytic set defined by the formula
	\eqref{eqn-X}. In particular, there is essentially only one free
	critical orbit on $\mathcal{X}$. 
  \item for any $P \in \mathcal{M}'$, there exists a quadratic-like
	restriction $P^{p'}:W_P' \to W_P$ hybrid equivalent to
	$Q=\xi(P)$ such that the map $\xi:\mathcal{M}' \to \mathcal{M}$
	is a homeomorphism.  The period $p$ (depending only on
	$\mathcal{M}'$)
	can be taken arbitrarily large.
  \item Let $P_1$ be the center of $\mathcal{M}'$, i.e., let $P_1$
	satisfies $\xi(P_1)=z^2$. 
	Then $\lambda_{P_1}$ is primitive.
 \end{enumerate}
\end{lem}

\subsection{Proof of Theorem~\ref{thm-mis-perturb}}

Let $\lambda_0$ be a post-critically finite $d$-invariant rational
lamination and let $f_0 \in \mathcal{R}(\lambda_0)$ be Misiurewicz.
Consider the algebraic set $X$ in Lemma~\ref{lem-thickening}.
If $f_0 \in X$, then perturb $f_0$ as in Section~\ref{subsec-perturb}
and we assume $f_0 \not \in X$.
(Precisely speaking, consider a one-parameter subfamily where all but one
critical orbit relation is preserved. By transversality,
\cite{vanStrien-transversality} we may assume $f_0$ is discrete in the
intersection of this subfamily and $X$,
so we can perturb $f_0$ on this subfamily).
Take a small neighborhood $\mathcal{V}$ of $f_0$.
We may assume $\mathcal{V} \cap \mathcal{C}(\lambda_0) \subset
\mathcal{R}(\lambda_0)$ by Lemma~\ref{lem-thickening}.
Take a neighborhood $\mathcal{U}$ of $P_0=\chi_{\lambda_0}(f_0)$
sufficiently small such that 
\begin{itemize}
 \item the codimension one algebraic set $Y$ in Theorem~\ref{thm-pcf-tuning}
       does not intersect $\mathcal{U}$ (if $P_0 \in Y$, we again
       perturb $f_0$ so that $P_0 \not \in Y$). 
       Therefore, for any post-critically finite $P \in \mathcal{U}$,
       there exists a unique $f$ such that $\chi_{\lambda_0}(f)=P$, and
 \item $\chi_{\lambda_0}^{-1}: \Mis(\poly(T(\lambda_0))) \cap
       \mathcal{U} \to \Mis(\poly(d)) \cap \mathcal{R}(\lambda_0)$
        is a homeomorphism into its image and the closure of the image
       is contained in $\mathcal{V}$.
\end{itemize}
The existence of such a neighborhood $\mathcal{U}$ is guaranteed by
Theorem~\ref{thm-conti}
and Proposition~\ref{prop-Mis-proper}.

Now apply Lemma~\ref{lem-P} for this $\mathcal{U}$.
Take a sequence of Misiurewicz polynomials $Q_n \in \partial
\mathcal{M}$ ($n \ge 2$) such that $Q_n \to Q_0(z)=z^2+1/4$ and let
$P_n=\xi^{-1}(Q_n)$.
(Recall that $P_1$ is the center of $\mathcal{M}'$.)
Then $P_n$ is also Misiurewicz for $n \ge 2$.
Let $f_n=\chi_{\lambda_0}^{-1}(P_n) \in \mathcal{V}$ for $n \ge 1$.

Let $\lambda=\lambda_{f_1}$ be the combinatorial tuning of $\lambda_0$
and $\lambda_{P_1}$. Since we may assume the period $p$ of
quadratic-like renormalization of $P_1$ is arbitrarily large,
$\lambda$ is also primitive by Lemma~\ref{lem-comb-tuning-primitive}.
Therefore, $\mathcal{C}(\lambda)=\mathcal{R}(\lambda)$ is compact.

As in Lemma~\ref{lem-P},
there exists a one-dimensional algebraic subset
\[
 \mathcal{Y}=\{(f,\omega_1,\dots,\omega_{d-1}) \in \widehat{\poly}(d);~
 \hat{h}_3(\omega_3)=\cdots =
 \hat{h}_{d-1}(\omega_{d-1})=\hat{h}(f,\omega_1,\omega_2)=0\}
\]
containing all $f_n$.
Since $\mathcal{R}(\lambda)$ is compact,
we may assume that $f_n$ converges to some $f \in \mathcal{R}(\lambda)$.
Then
\[
 f \in \mathcal{C}(\lambda) \cap \mathcal{V} \subset
 \mathcal{C}(\lambda_0) \cap \mathcal{V} =\mathcal{R}(\lambda_0) \cap
 \mathcal{V}.
\]
Namely, $f$ is $\lambda_0$-renormalizable and close to $f_0$.

Since $\mathcal{Y}$ is closed, $f$ (precisely speaking,
$(f,\omega_1,\dots,\omega_{d-1})$) also lies in $\mathcal{Y}$.
Take $w_j \in |T(\lambda_0)|$ such that $\omega_j \in K_f(w_j)$ for
$j=1,2$ and let
\begin{align*}
 N &= \sum_{j=1}^{N'-1}\ell_{\sigma^j(w_2)}, &
 p &= \sum_{j=1}^{p'-1}\ell_{\sigma^j(w_1)}.
\end{align*}
Then it follows that $f^p(\omega_1)=f^N(\omega_2)$.
By this relation and Theorem~\ref{thm-quad-conti}, 
we have $\chi_{\lambda}(f) = \tilde{Q}_0$, where 
\[
 \tilde{Q}_0(v_j,z) = 
  \begin{cases}
   (v_0,Q_0(z)) & \mbox{when }j=0, \\
   (v_0,z^2+Q_0(0)) & \mbox{when }j=1.
  \end{cases}
\]
It is easy to check that $f$ satisfies \condII\ 
(note that $p$ above is different from that in \condI).
\qed

\section{Discontinuity}
\label{sec-discont}

Now we give a proof of the main theorem:

\begin{proof}[Proof of Main Theorem]
 First, observe that there always exists a Misiurewicz polynomial $f_0
 \in \mathcal{R}(\lambda_0)$ assuming that $\mathcal{R}(\lambda_0)$ is
 nonempty, by Theorem~\ref{thm-combcont} and Theorem~\ref{thm-pcf-tuning}.

 Assume that $\chi_{\lambda_0}$ is continuous on $\mathcal{V} \cap
 \mathcal{R}(\lambda_0)$ for a neighborhood $\mathcal{V} \subset \poly(T_0)$
 of $f_0$.

 By Theorem~\ref{thm-mis-perturb} and Theorem~\ref{thm-para-purturb},
 there exists $\tilde{f} \in \mathcal{V} \cap
 \mathcal{R}(\lambda_0)$ satisfying \condIII.
 In the following, we use the notations in \condIII\ like $\omega,\ \omega',\
 p,\ N,\ V$ and $V'$ for $\tilde{f}$.
 Let $w_0, w_1 \in |T(\lambda_0)|$ satisfy $\omega \in K_{\tilde{f}}(w_0)$ and
 $\omega' \in K_{\tilde{f}}(w_1)$.
 Let $s'$ be the period of $w_0$ by $\sigma$, and 
 $s$ be the period of $K_{\tilde{f}}(w_0)$, in other words,
 \[
  s=\sum_{n=0}^{s'-1} \ell_{\sigma^n(w_0)}.
 \]
 Similarly, define $N'$ and $p'$ by 
 \begin{align*}
  \sum_{n=0}^{N'-1}\ell_{\sigma^n(w_1)}&=N, &
  \sum_{n=0}^{p'-1}\ell_{\sigma^n(w_0)}&=p.
 \end{align*}
 Observe that $K(\tilde{f}^p;V', V) \subset K_{\tilde{f}}(w_0)$.
 In particular, $s'$ divides $s$ and $p'$ divides $p$.

 By shrinking $\mathcal{V}$ if necessary,
 we may assume any $f \in \mathcal{V}$ has a polynomial-like restriction
 $g_f =(f^{\ell_w}:U_{g,w}' \to U_{g,\sigma(w)})_{w \in
 |T(\lambda_0)|}$ over $T(\lambda_0)$ such that
 \begin{enumerate}
  \item $g_f$ is a $\lambda_0$-renormalization when $f \in
	\mathcal{R}(\lambda_0)$.
  \item $(f^s: U_{f,w_0}'' \to U_{f,w_0})_{f \in \mathcal{V}}$ forms an AFPL,
	where $U_{f,w_0}''$ is the component of $f^{-s}(U_{f,w_0})$ containing
	$K(g_f,w_0)$.
 \end{enumerate}
 It follows by Theorem~\ref{thm-mis-perturb} that
 $K(\tilde{f}^p;V',V) \subset K_{\tilde{f}}(w_0)$.
 Observe that by definition,
 $g_f^{s'} = f^s$ and $g_f^{p'} = f^p$ on $K_f(w_0)$, and
 $g_f^{N'} = f^N$ on $K_f(w_1)$.

 By taking a finite branched cover of $\mathcal{V}$ if necessary,
 we may assume there exist analytic parameterizations of critical points
 $\omega(f)$ and $\omega'(f)$ such that $\omega(\tilde{f})=\omega$ and
 $\omega'(\tilde{f})=\omega'$.
 For $f \in \mathcal{R}(\lambda_0) \cap \mathcal{V}$, 
 let $P_f = \chi_{\lambda_0}(f) \in \mathcal{C}(T(\lambda_0))$ and 
 $\psi_f=(\psi_{f,w})_{w \in |T(\lambda_0)|}$ be a hybrid conjugacy
 between $g_f$ and $P_f$ 
 (we can take such a hybrid conjugacy $\psi_f$ by shrinking
 $U_{f,w}$ if necessary).
 Let
 \begin{align*}
  \omega(P_f)  &= (w_0,\psi_{f,w_0}(\omega(f))), & 
  \omega'(P_f) &= (w_1,\psi_{f,w_1}(\omega'(f)))
 \end{align*}
 be the
 critical points for $P_f$ corresponding to $\omega(f)$ and
 $\omega'(f)$ respectively.
 Let $x(f)=f^p(\omega(f))$ and $y(f)=f^N(\omega'(f))$
 and define $x(P_f)$ and $y(P_f)$ by
 \begin{align*}
  (w_0, x(P_f)) &= P_f^{s'}(\omega(P_f)), &
  (w_0, y(P_f)) &= P_f^{N'}(\omega'(P_f)).
 \end{align*}
 Observe that $x(P_f), y(P_f) \in K(P_f,w_0)$.
 Now consider an AFPL2MP 
 \[
  \h=(f^s:U_f'' \to U_f,x(f),y(f))_{f \in \mathcal{V}}.
 \]
 Then the straightening map $\chi_\h$ for $\h$ satisfies
 \[
  \chi_\h(f)=(\hat{P}_{f},
  \psi_{f,w_0}(f^p(\omega(f))),
  \psi_{f,w_0}(f^N(\omega'(f))) )
  = (\hat{P}_{f}, x(P_{f}), y(P_{f}))
 \]
 where $P_{f}^{s'}(w_0,z)=(w_0,\hat{P}_{f}(z))$.
 Since  $\chi_{\lambda_0}$ is continuous on $\mathcal{N}$,
 $\chi_\h$ is also continuous.
 
 Consider a repelling periodic point $\alpha=\alpha(f)$ in the filled
 Julia set $K(\tilde{f}^p;V',V)$ of the quadratic-like restriction $\tilde{f}^p:V'
 \to V$.
 Then we can take $f_{n,m}, f_n \in \mathcal{R}(\lambda_0) \cap
 \mathcal{V}$
 satisfying the conditions in \condIII.
 Therefore, we can apply Theorem~\ref{thm-conti-mult}, namely, we have
 \begin{equation}
  \label{eqn-mult}
  |\mult_{\tilde{f}}(\alpha)|=|\mult_{P_{\tilde{f}}}(\psi_{\tilde{f},w_0}(\alpha))|.
 \end{equation}
 
 Observe that $\psi_{\tilde{f}}$ is also a hybrid conjugacy from $\tilde{f}^p:V' \to
 V$ to a quadratic-like restriction of $P_{\tilde{f}}^{p'}$.
 Since \eqref{eqn-mult} holds for any repelling periodic point $\alpha
 \in K(\tilde{f}^p;V',V)$, it follows that $\psi_{\tilde{f}}|_V$ preserves multipliers.
 Therefore, by Theorem~\ref{thm-mult-semiconj}, 
 $\tilde{P}_{\tilde{f}}$ and $\tilde{f}^p$ are \semiconj,
 where $\tilde{P}_{\tilde{f}}$ is defined by 
 $P_{\tilde{f}}^{p'}(w_0,z)=(w_0,\tilde{P}_{\tilde{f}}(z))$.
 In particular, $\deg \tilde{P}_{\tilde{f}}=\deg \tilde{f}^p$.
 However, since $\lambda_0$ is nontrivial, we have $\deg P_{\tilde{f},w} <
 \deg \tilde{f} \le \deg \tilde{f}^{\ell_w}$
 for all $w$, so $\deg \tilde{P}_{\tilde{f}} < (\deg \tilde{f})^p$, that
 is a contradiction.
 Therefore, $\chi_{\lambda_0}$ is not continuous on $\mathcal{V}$.
\end{proof}

\begin{rem}
 \label{rem-discont-at-f_n}
 More precisely, we have proved the following:
 for any repelling periodic point $\alpha \in K(\tilde{f}^p;V',V)$ such
 that \eqref{eqn-mult} does not hold (such a repelling periodic point
 always exists), there exists a double sequence $f_{n,m} \to f_n \to \tilde{f}$
 satisfying the conditions in \condIII\ such that
 \[
 \lim_{m \to \infty} \chi_{\lambda_0}(f_{n,m}) \ne
 \chi_{\lambda_0}(f_n)
 \]
 for sufficiently large $n$, because $\chi_{\lambda_0}(f_n) \to
 \chi_{\lambda_0}(\tilde{f})$ by the quasiconformal rigidity of $\tilde{f}$.
\end{rem}


\section{The case of rational and transcendental entire maps}
\label{sec-rat-trans}

We do not know very much how rich the dynamics in a renormalizable
set is for families of rational maps and transcendental entire maps.
However, since the target space of a straightening map is a family of
polynomials over a mapping schema, we can apply the same argument to
obtain the following:

\begin{thm}
 \label{thm-rational}
 Let $(f_\mu)_{\mu \in \Lambda}$ be an analytic family of
 rational maps of degree $d \ge 3$.
 Assume there exists an (externally marked) AFPL
 $\g = (g_\mu = (f_\mu^{\ell_v}:U_v \to U_{\sigma(v)})_{v \in
 |T|})_{\mu \in \lambda}$ over a mapping schema
 $T=(|T|,\sigma,\delta)$ having a non-trivial critical relation.
 Let $\chi:\mathcal{C}(\g) \to \mathcal{C}(T)$ be the straightening map
 for $\g$. 
 For a Misiurewicz map $P_0 \in \mathcal{C}(T)$, assume there exist
 a neighborhood $\mathcal{U}$ of $P_0$ and a map
 $s:\mathcal{U} \cap \mathcal{C}(T) \to \mathcal{C}(\g)$
 such that $\chi \circ s$ is the identity.
 Then $s$ is not continuous,
 except when $(f_\mu)$ is affinely conjugate to a family of
 polynomials and $\delta(v)=d$ for all $v \in |T|$.

 In particular, there is no homeomorphic restriction of $\chi$ onto
 $\mathcal{U} \cap \mathcal{C}(T)$.
\end{thm}

An (externally marked) AFPL over a mapping schema, its connectedness
locus and its straightening map are defined in the same way.

\begin{proof}
 Let $P_1 \in \mathcal{C}(T) \cap \mathcal{U}$ satisfy \condIII\ and let
 $f_1 = s(P_1) \in \mathcal{C}(g)$.
 Then the same argument as Theorem~\ref{thm-discont} can be applied to
 $s^{-1}$ to show the discontinuity.
\end{proof}

\begin{thm}
 Let $(f_\mu)_{\mu \in \Lambda}$ be an analytic family of
 transcendental entire maps of degree $d \ge 3$.
 Assume there exists an (externally marked) AFPL
 $\g = (g_\mu = f_\mu^{\ell_v}:U_v \to U_{\sigma(v)})_{v \in |T|})_{\mu
 \in \Lambda}$ over a 
 mapping schema $T=(|T|,\sigma,\delta)$ having a non-trivial critical
 relation.
 Let $\chi:\mathcal{C}(\g) \to \mathcal{C}(T)$ be the straightening map
 for $\g$. 
 Let $P_0 \in \mathcal{C}(T)$ be Misiurewicz and assume there exist
 a neighborhood $\mathcal{U}$ of $P_0$ and a continuous map
 $s:\mathcal{U} \cap \mathcal{C}(T) \to \mathcal{C}(\g)$
 such that $\chi \circ s$ is the identity.

 Then there exist some $P_1 \in \mathcal{U} \cap \mathcal{C}(T)$
 satisfying \condIII,
 a polynomial $g$, $\varphi_1$ and a transcendental entire map
 $\varphi_2$ such that 
 \begin{align*}
  P_1 \circ \varphi_1 &= \varphi_1 \circ g, &
  f_1 \circ \varphi_2 &= \varphi_w \circ g,
 \end{align*}
 where $f_1=s(P_1)$.
\end{thm}

The proof is the same as Theorem~\ref{thm-rational}.
The only difference is that we cannot get a contradiction
after applying Theorem~\ref{thm-semiconj},
because the degree of a transcendental entire map is infinite 
and we cannot exclude the case in the conclusion.
Note that it follows that $g$ and $\varphi_1$ are polynomials by
comparing the growth at the infinity (see \cite{Inou-semiconj}).


\bibliographystyle{amsalpha}
\bibliography{ref}

\end{document}